\documentclass[10pt, reqno]{amsart}
\usepackage{amsmath, amssymb,amscd, amsthm,tikz-cd,mathrsfs}
\usepackage[margin=.8in]{geometry}
\usepackage[vcentermath]{youngtab}
\usepackage{hyperref, calligra, extarrows}
\usepackage{enumitem}

\newcommand{\E}{\text{\calligra om}\,}

\newcommand{\Addresses}{{
  \bigskip
  \footnotesize

  \textsc{Department of Mathematics, University of Notre Dame,
    Notre Dame, IN 46556}\par\nopagebreak
  \textit{E-mail address}: \texttt{mperlman@nd.edu}

}}

\newcommand{\tn}{\textnormal}
\newcommand{\D}{\mathcal{D}}
\newcommand{\C}{\mathbb{C}}
\newcommand{\bw}{\bigwedge}
\newcommand{\bS}{\mathbb{S}}

\newtheorem {theorem} {Theorem}

\newtheorem {lemma}[theorem] {Lemma}

\newtheorem {proposition}[theorem] {Proposition}

\theoremstyle{remark}
\newtheorem{remark}[theorem] {Remark}

\numberwithin{theorem}{section}
\numberwithin{equation}{section}

\title{Lyubeznik numbers for Pfaffian Rings}
\author{Michael Perlman}

\begin{document}
\begin{abstract}
We study the structure of local cohomology with support in Pfaffian varieties as a module over the Weyl algebra $\D_X$ of differential operators on the space of skew-symmetric matrices $X=\bw^2\C^n$. The simple composition factors of these modules are known by the work of Raicu-Weyman, and when $n$ is odd, the general theory implies that the local cohomology modules are semi-simple. When $n$ is even, we show that the local cohomology is a direct sum of indecomposable modules coming from the pole order filtration of the Pfaffian hypersurface. We then determine the Lyubeznik numbers for Pfaffian rings by computing local cohomology with support in the origin of the indecomposable summands referred to above.
\end{abstract}

\maketitle
\section{Introduction}

Let $X=\bw^2\C^n$ be the space of $n\times n$ skew-symmetric matrices over the complex numbers, endowed with the natural action of $\tn{GL}=\tn{GL}_n(\C)$. This action preserves rank, and for $0\leq k\leq \lfloor n/2 \rfloor$, the group $\tn{GL}$ acts transitively on the set of matrices $O_k$ of rank $2k$ (skew-symmetric matrices have even rank). The orbit closures $\overline{O}_k$ are the Pfaffian varieties of matrices of rank $\leq 2k$, and the goal of this work is to compute the Lyubeznik numbers of the local ring of each Pfaffian variety at the cone point. Lyubeznik numbers are numerical invariants associated to local rings \cite[Section 4]{lyubeznik1993finiteness}, and they may be defined in our case as follows. Let $S=\C[X]\cong \C[x_{i,j}]_{1\leq i<j\leq n}$ be the ring of polynomial functions on $X$ with homogeneous maximal ideal $\mathfrak{m}\subseteq S$, and let $E$ be the injective hull of $S/\mathfrak{m}$. Writing $R^k=(\C[\overline{O}_k])_{\mathfrak{m}}$ for the localization of the homogeneous coordinate ring of the Pfaffian variety $\overline{O}_k$, the Lyubeznik numbers $\lambda_{i,j}(R^k)$ are determined by the following composition of local cohomology functors:
\begin{equation}\label{lyubeznik}
H^i_{\{0\}}H^{\binom{n}{2}-j}_{\overline{O}_k}(S)=E^{\oplus \lambda_{i,j}(R^k)},
\end{equation}
where $\binom{n}{2}=\dim X$. For more information on Lyubeznik numbers in general, see the survey \cite{nunez2016survey}.

We compute the compositions (\ref{lyubeznik}) by studying the $\D_X$-module structure of $H^{\bullet}_{\overline{O}_k}(S)$, where $\D_X$ is the Weyl algebra of differential operators on $X$ with polynomial coefficients. In general, local cohomology of the polynomial ring is a holonomic (and thus finite length) $\D$-module, and the simple composition factors in the case of Pfaffian varieties are known by Raicu-Weyman \cite[Main Theorem]{raicu2016local}. We expand upon their work by describing the filtrations of these modules in the category $\tn{mod}_{\tn{GL}}(\D_X)$ of $\tn{GL}$-equivariant coherent $\D$-modules on $X$. In this category, the simple objects are in one-to-one correspondence with orbits $O_k$, and we denote them by $E=D_0, D_1,\cdots, D_{\lfloor n/2 \rfloor}=S$. When $n=2m+1$ is odd, the category $\tn{mod}_{\tn{GL}}(\D_X)$ is semi-simple \cite[Theorem 5.7(b)]{lHorincz2018categories}, and the $\D_X$-module structure of $H^j_{\overline{O}_k}(S)$ for $j\geq 0$ is completely determined by the previous work \cite{raicu2016local}. When $n=2m$ is even, the category $\tn{mod}_{\tn{GL}}(\mathcal{D}_X)$ is no longer semi-simple \cite[Theorem 5.7(a)]{lHorincz2018categories}. In particular, the localization $S_{\tn{Pf}}$ at the $n\times n$ Pfaffian is not semi-simple, and has composition series as follows \cite[Section 6]{raicu2016characters}:
\begin{equation}\label{filtration}
0\subset S \subsetneq \langle \textnormal{Pf}^{-2} \rangle_{\mathcal{D}}\subsetneq \langle \textnormal{Pf}^{-4} \rangle_{\mathcal{D}}\subsetneq \cdots \subsetneq \langle \textnormal{Pf}^{-(n-2)} \rangle_{\mathcal{D}}\subsetneq S_{\textnormal{Pf}},
\end{equation}
where $\langle \tn{Pf}^{\; d}\rangle_{\D}$ is the $\D_X$-submodule of $S_{\tn{Pf}}$ generated by the $d$-th power of Pf, and $\langle \tn{Pf}^{-2s}\rangle_{\D}/\langle \tn{Pf}^{-2s+2}\rangle_{\D}\cong D_{m-s}$. Set $Q_m=S_{\tn{Pf}}$, and for $p=0,\cdots, m-1$ we define
\begin{equation}\label{Q}
Q_p=\frac{S_{\tn{Pf}}}{\langle \tn{Pf}^{\;2(p-m+1)}\rangle_{\D}}.
\end{equation}
For each $p=0,\cdots ,m$, the $\mathcal{D}_X$-module $Q_p$ is indecomposable (see \cite[Theorem 5.7(a)]{lHorincz2018categories}, \cite[Lemma 6.3]{lHorincz2018iterated}), and has composition factors $D_0,\cdots , D_p$, each with multiplicity one.

\begin{theorem}\label{structure}
Let $n=2m$ be even. For every $0\leq k\leq m-1$ and $j\geq 0$ the local cohomology modules $H^j_{\overline{O}_k}(S)$ are direct sums of the modules $Q_0, Q_1, \cdots, Q_{k}$.
\end{theorem}
\noindent We make the statement of the theorem more precise as follows. Write $\Gamma_{\D}$ for the Grothendieck group of the category $\tn{mod}_{\tn{GL}}(\D_X)$. Using \cite[Main Theorem]{raicu2016local}, and applying binomial identities in Lemma \ref{Qdecomp}, we obtain the following formulas in $\Gamma_{\D}[q]$ for $n=2m$: $\sum_{j\geq 0} [H^j_{\overline{O}_{m-1}}(S)]_{\D}\cdot q^j=[Q_{m-1}]_{\D}\cdot q$ and for $0\leq k\leq m-2$
\begin{equation}\label{RW}
\sum_{j\geq 0}\left[H^j_{\overline{O}_k}(S)\right]_{\D}\cdot q^j=\sum_{p=0}^k [Q_p]_{\D}\cdot q^{2(m-k)^2-(m-k)+4(k-p)}\cdot \binom{m-p-2}{k-p}_{q^4},
\end{equation}
where $\binom{a}{b}_{q^4}$ is a Gaussian binomial coefficient (see Section \ref{Rep}). Since $[Q_0]_{\D}, \cdots ,[Q_m]_{\D}$ form a basis for $\Gamma_{\D}$, Theorem \ref{structure} says that formula (\ref{RW}) completely describes the indecomposable summands of the local cohomology.

The problem of computing the Lyubeznik numbers (\ref{lyubeznik}) when $n=2m$ now reduces to determining the number of copies of the simple module $E$ inside $H^i_{\{0\}}(Q_p)$ for all $i\geq 0$ and $0\leq  p\leq m$. When $n=2m+1$ is odd, semi-simplicity of the category $\tn{mod}_{\tn{GL}}(\D_X)$ implies that, in order to compute the Lyubeznik numbers using \cite[Main Theorem]{raicu2016local}, it suffices to determine the number of copies of $E$ in the local cohomology $H^i_{\{0\}}(D_s)$ for all $i\geq 0$ and all $s$. We encode our Lyubeznik number computations via generating functions $L_k(q,w)\in \mathbb{Z}[q,w]$:
\begin{equation}
L_k(q,w)=\sum_{i,j\geq 0} \lambda_{i,j}(R^k)\cdot q^i\cdot w^j.
\end{equation}
We are now ready to state the main theorem. 
\begin{theorem}\label{main}
Let $m=\lfloor n/2 \rfloor$. For $n$ even, $L_{m-1}(q,w)=(q\cdot w)^{\binom{n}{2}-1}$. Otherwise, for $0\leq k<m$ the Lyubeznik numbers for $R^k$ are given by the following formulas:
\begin{equation}\label{lyubeznikForm}
L_k(q,w)=
\begin{cases}
\displaystyle\sum_{s=0}^k q^{s(2s+3)}\cdot \binom{m-1}{s}_{q^4}\cdot w^{k(2k+3)-4s(k-m+1)}\cdot \binom{m-s-2}{k-s}_{w^4} & \tn{ if $n=2m$ is even,}\\
\displaystyle\sum_{s=0}^k q^{s(2s+1)}\cdot \binom{m}{s}_{q^4}\cdot w^{k(2k+3)-2s(2k-2m+1)}\cdot \binom{m-s-1}{k-s}_{w^4} & \tn{ if $n=2m+1$ is odd.}
\end{cases}
\end{equation}
\end{theorem}
\noindent We remark that, for all $n\geq 2$, the Pfaffian variety $\overline{O}_1$ is the affine cone over the Grassmannian $\tn{Gr}(2,\mathbb{C}^n)$, a smooth projective variety. Thus, the Lyubeznik numbers $\lambda_{i,j}(R^1)$ are determined by the algebraic de Rham cohomology of the Grassmannian \cite[Main Theorem 1.2]{switala2015lyubeznik}. Using said theorem, when $n=6$, the only nonvanishing Lyubeznik numbers are $\lambda_{0,5}(R^1)=\lambda_{5,9}(R^1)=\lambda_{9,9}(R^1)=1$. Alternatively, we may compute these using Theorem \ref{main}:
$$
L_1(q,w)=\sum_{s=0}^1 q^{s(2s+3)}\cdot \binom{2}{s}_{q^4}\cdot w^{5+4s}=w^5+q^5\cdot w^9+q^9\cdot w^9,
$$
recovering the previous Lyubeznik number computations in this case.

We discuss parallels between this work and the recent work of L\H{o}rincz-Raicu \cite{lHorincz2018iterated}, where the authors compute all iterations of local cohomology with support in generic determinantal varieties, yielding the Lyubeznik numbers for determinantal rings. The dichotomy in our case between even-sized and odd-sized skew-symmetric matrices is analogous to the dichotomy between square and non-square generic matrices. For integers $n\neq m$, the category of $\tn{GL}_n(\C)\times \tn{GL}_m(\C)$-equivariant coherent $\D$-modules on the space of $n\times m$ generic matrices is semi-simple, similar to the odd-sized skew-symmetric case. Also, just as in Theorem \ref{structure}, the local cohomology modules of the polynomial ring with support in determinantal varieties of square matrices decompose as a direct sum of analogues of the $Q_p$'s coming from the pole order filtration of the determinant hypersurface \cite[Theorem 1.6]{lHorincz2018iterated}. As such, our technique for proving Theorem \ref{structure} is similar to the proof of \cite[Theorem 1.6]{lHorincz2018iterated}. However, our methods for computing the local cohomology of the indecomposable modules with support in $\{0\}$ are different from the methods used in \cite{lHorincz2018iterated}. The authors use syzygy computations in their Section 3.2 for which the analogous results in the skew-symmetric case are unknown. In order to compute local cohomology $H^{\bullet}_{\{0\}}(D_s)$ and $H^{\bullet}_{\{0\}}(Q_p)$ in the case of even-sized skew-symmetric matrices, we instead use previous computations involving $\tn{Ext}^{\bullet}_S(S/I,S)$, where $I\subset S$ is a $\tn{GL}$-invariant ideal (see Section \ref{Perlman}). To compute $H^i_{\{0\}}(D_s)$ in the case when $n=2m+1$ is odd, we construct a birational isomorphism to $X$ from a space that locally looks like the product of the space of $2m\times 2m$ skew-symmetric matrices with an affine space. This allows us to ``push forward" the computations from the even case to the odd case.\\

\noindent \textbf{Organization.} In Section \ref{Prelim} we review the necessary background on representation theory, Gaussian binomial coefficients, $\D$-modules, and $\tn{GL}_n(\C)$-invariant ideals on $\bw^2\mathbb{C}^n$. In Section \ref{Local} we prove Theorem \ref{structure}, which we use in Section \ref{Even} to calculate the Lyubeznik numbers for Pfaffian rings of even-sized skew-symmetric matrices. Finally, we obtain the Lyubeznik numbers for the case of odd-sized skew-symmetric matrices in Section \ref{oddSect}.

\section{Preliminaries}\label{Prelim}
\subsection{Representation theory, Gaussian binomial coefficients, and Bott's theorem for Grassmannians}\label{Rep}
 Let $W$ be an $n$-dimensional complex vector space. The irreducible representations of $\tn{GL}=\text{GL}(W)$ are in one-to-one correspondence with dominant weights $\mu=(\mu_1\geq \cdots \geq \mu_n)\in \mathbb{Z}^n$. We denote the set of dominant weights with $n$ parts by $\mathbb{Z}^n_{\text{dom}}$. Given a dominant weight $\mu$, write $\bS_{\mu}W$ for the irreducible representation of highest weight $\mu$, where $\bS_{\mu}(-)$ is the Schur functor. If $M$ is a representation of $\tn{GL}$ which decomposes as a direct sum of irreducibles with finite multiplicity, we write $\langle M, \bS_{\lambda}W\rangle$ for the multiplicity of $\bS_{\lambda}W$ as a summand in $M$. When $\mu=(1,\cdots , 1)=(1^n)$, the representation $\bS_{\mu}W$ is equal to $\bw^n W$, which we will sometimes write as $\text{det}(W)$. Similarly, when $\mathcal{E}$ is a locally free sheaf, we sometimes write $\tn{det}(\mathcal{E})$ for the highest exterior power. Given dominant weights $\lambda, \mu\in \mathbb{Z}^n_{\tn{dom}}$, write $\lambda\geq \mu$ if $\lambda_i\geq \mu_i$ for all $1\leq i\leq n$. The size of a weight is $|\mu|=\mu_1+\cdots +\mu_n$. When $\mu_n\geq 0$, we refer to $\mu$ as a partition, and we will use underlined roman letters to distinguish partitions from arbitrary dominant weights. Denote by $\mathcal{P}(n)$ the set of partitions in $\mathbb{Z}^n_{\text{dom}}$. It is convenient to identify a partition $\underline{x}$ with the associated Young diagram:
$$
\underline{x}=(4,3,1,0)\;\; \longleftrightarrow \;\;\;\yng(4,3,1)
$$
\vspace{.01cm}

\noindent When we refer to a row or column of a partition $\underline{x}$, we mean the corresponding row or column of its associated Young diagram. Given $\lambda\in \mathbb{Z}^k_{\tn{dom}}$ we write
$$
\lambda^{(2)}=(\lambda_1,\lambda_1,\lambda_2,\lambda_2,\cdots)\in \mathbb{Z}^{2k}_{\tn{dom}}.
$$
For $\underline{z}\in \mathcal{P}(k)$, the Young diagram of $\underline{z}^{(2)}$ is obtained from the Young diagram of $\underline{z}$ by doubling the column lengths. When $\underline{x}$ is a rectangle, i.e. $\underline{x}=(b,\cdots ,b,0,\cdots ,0)=(b^a)$ for some $a,b\in \mathbb{N}$, we will sometimes write $\underline{x}=a\times b$. 

Given integers $a\geq b\geq 0$, the Gaussian binomial coefficient $\binom{a}{b}_q$ is a polynomial in $\mathbb{Z}[q]$ defined via
\begin{equation}
\binom{a}{b}_q=\frac{(1-q^a)\cdot (1-q^{a-1})\cdots (1-q^{a-b+1})}{(1-q^b)\cdot (1-q^{b-1})\cdots (1-q)},
\end{equation}
with the convention that $\binom{a}{a}_q=\binom{a}{0}_q=1$. These polynomials will be used throughout to state our computations concisely. In the proof of Lemma \ref{multQuot}, we will use that Gaussian binomial coefficients encode the number of partitions $\underline{x}$ of a given size that fit inside an $a\times b$ rectangle:
\begin{equation}\label{binomPart}
\binom{a+b}{b}_q=\sum_{\underline{x}\leq (b^a)}q^{|\underline{x}|}.
\end{equation}
In other words, the coefficient of $q^j$ in $\binom{a+b}{b}_q$ is the number of partitions of size $j$ that fit inside an $a\times b$ rectangle. In order to simplify our formulas, we will repeatedly use the binomial identity
\begin{equation}\label{binomIdent}
\binom{a}{b}_q=\binom{a-1}{b-1}_q+q^b\cdot \binom{a-1}{b}_q,\;\;\tn{ for $a>b>0$.}
\end{equation}
Finally, we recall the identity
\begin{equation}\label{binomInvert}
\binom{a}{b}_{q^{-1}}=q^{-b(a-b)}\cdot \binom{a}{b}_q.
\end{equation}
This will be used in the proof of Lemma \ref{swapOrder} and Lemma \ref{locPfaff}.

Let $V$ be an $n$-dimensional vector space, let $\mathbb{G}=\tn{Gr}(k,V)$ be the Grassmannian of $k$-dimensional subspaces of $V$, and let $\mathcal{R}$ and $\mathcal{Q}$ be the tautological sub and quotient sheaves of $V\otimes_{\mathbb{C}}\mathcal{O}_{\mathbb{G}}$ ($\tn{rank}(\mathcal{R})=k$, $\tn{rank}(\mathcal{Q})=n-k$). By Bott's theorem for Grassmannians \cite[Theorem 4.1.8]{weyman2003cohomology}, there is a method for computing the cohomology of bundles of the form $\bS_{\lambda}\mathcal{R}^{\ast}\otimes \bS_{\mu}\mathcal{Q}^{\ast}$ for $\lambda\in \mathbb{Z}^{k}_{\tn{dom}}$ and $\mu\in \mathbb{Z}^{n-k}_{\tn{dom}}$, which we now describe. Let $\rho=(n-1,n-2,\cdots, 1,0)$ and set $\gamma=(\lambda,\mu)=(\lambda_1,\cdots ,\lambda_k,\mu_1,\cdots, \mu_{n-k})$. Write $\tn{sort}(\rho+\gamma)$ for the element of $\mathbb{Z}^n$ obtained by putting the entries of $\rho+\gamma$ in non-increasing order, and write $\tilde{\gamma}=\tn{sort}(\rho+\gamma)-\rho$. If we write $\sigma$ for the minimal number of transpositions required to put $\rho+\gamma$ in non-increasing order, then
\begin{equation}\label{Bott}
H^j(\mathbb{G},\bS_{\lambda}\mathcal{R}^{\ast}\otimes \bS_{\mu}\mathcal{Q}^{\ast})=
\begin{cases}
\bS_{\tilde{\gamma}}V^{\ast} & \tn{if $\gamma+\rho$ has distinct entries and $j=\sigma$,}\\
0 & \tn{otherwise.}
\end{cases}
\end{equation}
This notation is used in Section \ref{vanish} to prove Theorem \ref{structure}, and Bott's theorem is used to prove Lemma \ref{pushEven}.

\subsection{The $\tn{GL}$-invariant ideals, the modules $J_{\underline{z},l}$, and Ext computations}\label{Perlman}

Let $W$ be an $n$-dimensional complex vector space, and let $S=\tn{Sym}(\bw^2 W)$ be the ring of polynomial functions on the space of $n\times n$ skew-symmetric matrices. This ring has a natural action of $\tn{GL}=\tn{GL}(W)$ and decomposes into irreducibles as follows \cite[Proposition 2.3.8]{weyman2003cohomology}:
\begin{equation}\label{cauchy}
S=\bigoplus_{\underline{z}\in \mathcal{P}(m)} \bS_{\underline{z}^{(2)}} W,
\end{equation}
where $m=\lfloor n/2 \rfloor$. Given $\underline{z}\in \mathcal{P}(m)$, we write $I_{\underline{z}}\subset S$ for the ideal generated by $\bS_{\underline{z}^{(2)}} W$. Note that
\begin{equation}\label{charI}
I_{\underline{z}}=\bigoplus_{\underline{x}\geq \underline{z}} \bS_{\underline{x}^{(2)}} W,
\end{equation}
so that 
\begin{equation}\label{contain}
I_{\underline{x}}\subseteq I_{\underline{z}}\;\;\;\iff \;\;\; \underline{z}\leq \underline{x}.
\end{equation}
Under the identification $S\cong \mathbb{C}[x_{i,j}]_{1\leq i<j\leq n}$, when $\underline{x}=(1,\cdots,1,0,\cdots,0)=(1^k)=k\times 1$, we have
\begin{equation}
I_{k\times 1}=\left\langle 2k\times 2k\textnormal{ Pfaffians of $(x_{i,j})$}\right\rangle,
\end{equation}
where $(x_{i,j})$ is the skew-symmetric matrix of indeterminates. The orbit closure $\overline{O}_k\subseteq \bw^2 W^{\ast}$ is defined by the prime ideal $I_{(k+1)\times 1}$, and the dimension is given by:
\begin{equation}\label{dimension}
\dim \overline{O}_k=\dim S/I_{(k+1)\times 1}=k(2n-2k-1).
\end{equation}
When $n=2m$ is even and $k=m-1$, the orbit closure $\overline{O}_{m-1}$ is defined by the $n\times n$ Pfaffian $\tn{Pf}\in S$. Note that $\tn{Pf}$ has weight $(1^n)$. When referring to the fractional ideal $S\cdot \tn{Pf}^{\; e}\subseteq \tn{Frac}(S)$ with $e\in \mathbb{Z}$, we will write 
\begin{equation}\label{fractional}
\mathscr{P}_{e}=S\cdot \tn{Pf}^{\; e}\subseteq \tn{Frac}(S).
\end{equation}
These fractional ideals behave well with respect to the tensor product with an invariant ideal $I_{\underline{z}}$: it follows from (\ref{charI}) that, for $\underline{z}\in \mathcal{P}(m)$, we have
\begin{equation}\label{twister}
I_{\underline{z}} \otimes_S \mathscr{P}_e \cong \bigoplus_{\substack{\nu\in \mathbb{Z}^m_{\tn{dom}}\\ \nu\geq \underline{z}+(e^m)}} \bS_{\nu^{(2)}}W.
\end{equation}
In particular, if $e\geq 0$, then $I_{\underline{z}}\otimes_S \mathscr{P}_e=I_{\underline{z}+(e^m)}$.

We now recall the definition of the subquotients $J_{\underline{z},l}$ introduced in \cite[Lemma 2.5]{raicu2016local}. We note that our notation for these modules is consistent with \cite{perlman2017regularity}, and our $J_{\underline{z},l}$ is equal to $J_{\underline{z}^{(2)},l}$ in \cite{raicu2016local}. For $\mathcal{X}\subset \mathcal{P}(m)$, write $I_{\mathcal{X}}=\sum_{\underline{x}\in \mathcal{X}} I_{\underline{x}}$. Given $0\leq l\leq m-1$ and $\underline{z}\in \mathcal{P}(m)$ with $z_1=\cdots =z_{l+1}$, we consider the collection of partitions obtained from $\underline{z}$ by adding boxes to its Young diagram in row $l+1$ or higher:
\begin{equation}\label{succ}
\mathfrak{succ}(\underline{z},l,m)=\left\{ \underline{y}\in \mathcal{P}(m) \mid \underline{y}\geq \underline{z}\;\; \textnormal{and} \;\; y_i>z_i\;\; \textnormal{for some}\;\; i>l\right\}.
\end{equation}
We write
\begin{equation}\label{J}
J_{\underline{z},l}=I_{\underline{z}}/I_{\mathfrak{succ}(\underline{z},l,m)}.
\end{equation}

All of the facts in the remainder of the subsection are a consequence of \cite[Main Theorem]{perlman2017regularity}. Given $\underline{z}\in \mathcal{P}(m)$, the $S$-module $S/I_{\underline{z}}$ has a $\tn{GL}$-equivariant filtration by $S$-modules with successive quotients $J_{\underline{x},p}$ for various $\underline{x}\in \mathcal{P}(m)$ and $0\leq p\leq m-1$, and the pairs $(\underline{x},p)$ for which the module $J_{\underline{x},p}$ appears (all with multiplicity one) as a subquotient of $S/I_{\underline{z}}$ are described by the following set
\begin{equation}\label{mainExt}
\mathcal{Z}(\underline{z})=\left\{ (\underline{x},z'_{c+1}-1)\mid 0\leq c\leq z_1-1,x_1=\cdots=x_{z'_{c+1}}=c,\tn{ and }\underline{z}(c)\leq \underline{x}\right\}.
\end{equation}
This filtration induces a degenerate spectral sequence for computing $\tn{Ext}^{\bullet}_S(S/I_{\underline{z}},S)$, yielding the following isomorphism of representations of $\tn{GL}$ (but not necessarily an isomorphism of $S$-modules):
\begin{equation}\label{degen}
\tn{Ext}^j_S(S/I_{\underline{z}},S)\cong \bigoplus_{(\underline{x},p)\in \mathcal{Z}(\underline{z})} \tn{Ext}^j_S(J_{\underline{x},p},S).
\end{equation}
Thus, the $\tn{GL}$-structure of $\tn{Ext}^j_S(S/I_{\underline{z}},S)$ is determined by the $\tn{GL}$-structure of $\tn{Ext}^j_S(J_{\underline{x},p},S)$ for $(\underline{x},p)\in \mathcal{Z}(\underline{z})$, which is explicitly described in \cite[Theorem 3.2]{perlman2017regularity}. 

For the argument in Theorem \ref{evenSimp}, we will need information about morphisms $\tn{Ext}^j_S(S/I_{\underline{y}},S)\to \tn{Ext}^j_S(S/I_{\underline{z}},S)$ induced by an inclusion $I_{\underline{z}}\subseteq I_{\underline{y}}$. By \cite[Main Theorem]{perlman2017regularity}, we have the following method to compute the (co)kernels of these maps:
\begin{equation}\label{kernel}
\tn{ker}\left( \tn{Ext}^j_S(S/I_{\underline{y}},S)\longrightarrow \tn{Ext}^j_S(S/I_{\underline{z}},S)\right)\cong\bigoplus_{(\underline{x},p)\in\mathcal{Z}(\underline{y})\setminus \mathcal{Z}(\underline{z})} \tn{Ext}^j_S(J_{\underline{x},p},S),
\end{equation}
\begin{equation}\label{cokernel}
\tn{coker}\left( \tn{Ext}^j_S(S/I_{\underline{y}},S)\longrightarrow \tn{Ext}^j_S(S/I_{\underline{z}},S)\right)\cong\bigoplus_{(\underline{x},p)\in\mathcal{Z}(\underline{z})\setminus \mathcal{Z}(\underline{y})} \tn{Ext}^j_S(J_{\underline{x},p},S),
\end{equation}
where the above isomorphisms are in the category of $\tn{GL}$-modules. We store the following for use in Section \ref{Local}.

\begin{lemma}\label{multQuot}
Let $n=2m$ be an even integer. We have the following for $a,b\in \mathbb{N}$ with $b\geq 2a-1$:
$$
\sum_{j\geq 0}\left\langle \textnormal{Ext}^{j}_S(S/I_{a\times b},S), \tn{det}(W^{\ast})^{\otimes (n+b-2a)}\right\rangle\cdot q^j=q^{a(2a-3)-m(4a-2m-3)+1}\cdot \binom{m-1}{a-1}_{q^4}.
$$
\end{lemma}

\begin{proof}
By (\ref{mainExt}) and (\ref{degen}), it follows that 
\begin{equation}\label{sumJzl}
\left\langle \textnormal{Ext}^{\bullet}_S(S/I_{a\times b},S), \tn{det}(W^{\ast})^{\otimes (n+b-2a)}\right\rangle=\sum_{\substack{\underline{z}\in \mathcal{P}(m)\\ z_1=\cdots =z_a\leq b-1}} \left\langle \textnormal{Ext}^{\bullet}_S(J_{\underline{z},a-1},S), \tn{det}(W^{\ast})^{\otimes (n+b-2a)}\right\rangle.
\end{equation}
Thus, we need to compute the multiplicity of $\tn{det}(W^{\ast})^{\otimes (n+b-2a)}$ in $\tn{Ext}^j_S(J_{\underline{z},a-1},S)$ for all $\underline{z}\in \mathcal{P}(m)$ with $z_1=\cdots =z_a\leq b-1$. By \cite[Theorem 3.2]{perlman2017regularity} (referring the reader there for notation), we have 
$$
\tn{Ext}^j_S(J_{\underline{z},a-1},S)=\bigoplus_{\substack{ \mathcal{T}_{a-1}(\underline{z})\\ \binom{n}{2}-\binom{2a-2}{2}-2\sum_{i=1}^{n-2a+2}t_i=j\\ \lambda\in W(\underline{z},a-1,\underline{t})}} \bS_{\lambda}W^{\ast}.
$$
Given $\underline{t}\in \mathcal{T}_{a-1}(\underline{z})$, there exists $\lambda\in W(\underline{z},a-1,\underline{t})$ with $\bS_{\lambda}W^{\ast}=\tn{det}(W^{\ast})^{\otimes (n+b-2a)}$ if and only if
$$
\lambda_{2a-2+i-2t_i}=z_{2a-2+i}^{(2)}+n-1-2t_i=n+b-2a, \;\;\tn{for all $i=1,\cdots, n-2a+2$}.
$$
It follows that $z^{(2)}_{2a-2+i}$ must equal $b-2a+1+2t_i$ for all $i=1,\cdots, n-2a+2$. Since $t_1=a-1$, it follows that the multiplicity of $\tn{det}(W^{\ast})^{\otimes (n+b-2a)}$ in $\tn{Ext}^{\bullet}_S(J_{\underline{z},a-1},S)$ is zero unless $z_1=\cdots =z_a=b-1$. Now fix $\underline{z}\in \mathcal{P}(m)$ with $z_1=\cdots =z_a=b-1$. By the above discussion, we have that there exists a unique $\underline{t}\in \mathcal{T}_{a-1}(\underline{z})$ such that $((n+b-2a)^n)\in W(\underline{z},a-1,\underline{t})$, and thus
$$
\sum_{j\geq 0}\left\langle \tn{Ext}^j(J_{\underline{z},a-1},S), \tn{det}(W^{\ast})^{\otimes (n+b-2a)}\right\rangle\cdot q^j = q^{\binom{n}{2}-\binom{2a-2}{2}-2|\underline{t}|}.
$$
Given non-negative integers $c$ and $d$, write $\mathcal{P}(c;d)$ for the set of partitions $\underline{x}$ that fit inside a $c\times d$ rectangle ($\underline{x}\leq (d^c)$). We claim that for all $\underline{t}\in \mathcal{P}(n-2a+2;a-1)$ with $t_1=a-1$, there exists a unique $\underline{z}\in \mathcal{P}(m)$ with $z_1=\cdots =z_a=b-1$, $\underline{t}\in \mathcal{T}_{a-1}(\underline{z})$, and $((n+b-2a)^n)\in W(\underline{z},a-1,\underline{t})$. Indeed, given $\underline{t}\in \mathcal{P}(n-2a+2;a-1)$ with $t_1=a-1$, set $z_1=\cdots =z_a=b-1$ and $z_{a+i}=b-2a+1+2t_{2i}$ for $i=1,\cdots , m-a+1$. Then by the hypothesis on $b$ we have $\underline{z}\in \mathcal{P}(m)$, $\underline{t}\in \mathcal{T}_{a-1}(\underline{z})$, and $((n+b-2a)^n)\in W(\underline{z},a-1,\underline{t})$. It is clear that this choice of $\underline{z}$ was unique. By (\ref{sumJzl}) we conclude that
\begin{equation}\label{sumExtPrelim}
\sum_{j\geq 0}\left\langle \textnormal{Ext}^j_S(S/I_{a\times b},S), \tn{det}(W^{\ast})^{\otimes (n+b-2a)}\right\rangle\cdot q^j=\sum_{\substack{\underline{t}\in \mathcal{P}(n-2a+2;a-1)\\ t_1=a-1\\t_{2i}=t_{2i-1}}} q^{\binom{n}{2}-\binom{2a-2}{2}-2|\underline{t}|}
\end{equation}
There is a bijection between the set of $\underline{t}\in \mathcal{P}(n-2a+2;a-1)$ with $t_1=a-1$, $t_{2i}=t_{2i-1}$ for all $1\leq i\leq m-a+1$, and the set $\mathcal{P}(m-a;a-1)$ given by sending $\underline{t}$ to $(t_3,t_5,\cdots , t_{n-2a+1})$. Therefore, (\ref{sumExtPrelim}) is equal to 
$$
\sum_{\beta\in \mathcal{P}(m-a;a-1)}q^{\binom{n}{2}-\binom{2a-2}{2}-4(a-1)-4|\beta|}.
$$
Using (\ref{binomPart}), the result follows.
\end{proof}

\subsection{Characters of equivariant $\D$-modules on spaces of skew-symmetric matrices}\label{Dmodule}
Let $W$ be an $n$-dimensional complex vector space and let $S=\tn{Sym}(\bw^2 W)$. Write $\mathcal{D}$ for the ring of differential operators on $\bw^2W^{\ast}$, and write $\tn{mod}_{\tn{GL}}(\D)$ for the category of $\tn{GL}$-equivariant coherent $\D$-modules on $\bw^2W^{\ast}$. For ease of notation set $m=\lfloor n/2 \rfloor$. This category has $m+1$ simple objects $D_0=E, D_1,\cdots , D_m=S$, and via the Riemann-Hilbert correspondence the simple $D_s$ corresponds to the trivial local system on the orbit $O_s$ of skew-symmetric matrices of rank $2s$. When $n$ is even, the $n\times n$ Pfaffian $\tn{Pf}\in S$ has weight $(1^{n})$, so that by (\ref{cauchy}) we have
\begin{equation}\label{charLocal}
S_{\tn{Pf}}=\bigoplus_{\lambda\in \mathbb{Z}^m_{\tn{dom}}} \bS_{\lambda^{(2)}}W.
\end{equation}
The $\D$-module $S_{\tn{Pf}}$ has composition series (\ref{filtration}), with composition factors $D_0,\cdots, D_m$, each appearing with multiplicity one. Consider the sets
\begin{align*}
\mathcal{B}(s,2m) & = \left\{ \lambda\in \mathbb{Z}_{\textnormal{dom}}^{2m} \mid \lambda_{2s}\geq (2s-1), \lambda_{2s+1}\leq 2s, \lambda_{2i-1}=\lambda_{2i}\textnormal{ for all $i$} \right\},\\
\mathcal{B}(s,2m+1) & = \left\{ \lambda\in \mathbb{Z}_{\textnormal{dom}}^{2m+1} \mid \lambda_{2s+1}=2s, \lambda_{2i-1}=\lambda_{2i} \textnormal{ for }i\leq s, \lambda_{2i}=\lambda_{2i+1}\textnormal{ for $i>s$} \right\}.
\end{align*}
Then by \cite[Section 6]{raicu2016characters} we have that for $s=0,\cdots ,m$:
\begin{equation}\label{characters}
D_s\cong \bigoplus_{\lambda\in \mathcal{B}(m-s,n)} \bS_{\lambda}W^{\ast},
\end{equation}
In particular, $E=D_0$ is the only simple module containing the representation $\bS_{((n-1)^n)}W^{\ast}=\tn{det}(W^{\ast})^{\otimes (n-1)}$ and the multiplicity of this representation in $E$ is one. Note that, as a representation of $\tn{GL}$, we have $E=S^{\ast}\otimes_{\mathbb{C}}\tn{det}(W^{\ast})^{\otimes (n-1)}$, where $S^{\ast}=\oplus_{d\leq 0} (S_d)^{\ast}$. It is important to mention that our grading conventions on $E$ are different than the conventions on $^*E$ in \cite[Chapter 13, Chapter 14]{brodman1998local}. Let $d=\dim S=\binom{n}{2}$. In our case, $E_t=(S_{d+t})^{\ast}\otimes_{\mathbb{C}}\tn{det}(W^{\ast})^{\otimes (n-1)}$, so $E=\;^*E(d)$.

We will use the descriptions (\ref{characters}) in the proof of Lemma \ref{limitpfaff} and Section \ref{oddSect}. When $n=2m$ is even, the module $\langle \tn{Pf}^{-2k}\rangle_{\D}$ from the introduction has composition factors $S=D_m, D_{m-1},\cdots , D_{m-k}$, each with multiplicity one. Using (\ref{characters}), we obtain
\begin{equation}\label{charPf}
\langle \tn{Pf}^{-2k}\rangle_{\D}\cong \bigoplus_{\substack{\lambda\in \mathbb{Z}^{2m}_{\tn{dom}}\\ \lambda_{2k+1}\leq 2k\\ \lambda_{2i-1}=\lambda_{2i} \tn{for all $i$}}} \bS_{\lambda}W^{\ast}.
\end{equation}

\subsection{Some binomial identities}
In this subsection, we use the binomial identities in Section \ref{Rep} to simplify expressions that appear later in the article.
\begin{lemma}\label{Qdecomp}
Let $n=2m$ and $0\leq k\leq m-2$. The following formula holds in $\Gamma_{\D}[q]$:
$$
\sum_{j\geq 0}\left[H^j_{\overline{O}_k}(S)\right]_{\D}\cdot q^j=\sum_{p=0}^k [Q_p]_{\D}\cdot q^{2(m-k)^2-(m-k)+4(k-p)}\cdot \binom{m-p-2}{k-p}_{q^4}.
$$
\end{lemma}

\begin{proof}
As $[Q_p]_{\D}=\sum_{s=0}^p[D_s]_{\D}$, we have 
$$
\sum_{p=0}^k [Q_p]_{\D}\cdot q^{2(m-k)^2-(m-k)+4(k-p)}\cdot \binom{m-p-2}{k-p}_{q^4}=\sum_{s=0}^k \left([D_s]_{\D}\cdot q^{2(m-k)^2-(m-k)}\cdot  \sum_{p=s}^k q^{4(k-p)}\binom{m-p-2}{k-p}_{q^4}\right).
$$
Using the binomial identity (\ref{binomIdent}), it follows that for $k\leq m-2$:
$$
 \sum_{p=s}^k q^{4(k-p)}\binom{m-p-2}{k-p}_{q^4}=\sum_{p=s}^{k-1}\left( \binom{m-p-1}{k-p}_{q^4}-\binom{m-p-2}{k-p-1}_{q^4}\right)+\binom{m-k-2}{0}_{q^4}=\binom{m-s-1}{k-s}_{q^4}.
$$
By \cite[Main Theorem]{raicu2016local}, the result follows.
\end{proof}

\begin{lemma}\label{swapOrder}
Let $m=\lfloor n/2 \rfloor$ and let $d=\binom{n}{2}$. The following holds in $\Gamma_{\D}[q]$:
$$
\sum_{0\leq j\leq d} \left[H^{d-j}_{\overline{O}_k}(S)\right]_{\D}\cdot q^{j}=
\begin{cases}
\displaystyle \sum_{p=0}^k [Q_p]_{\D}\cdot q^{k(2k+3)-4p(k-m+1)}\cdot \binom{m-p-2}{k-p}_{q^4} & \tn{ for $n$ even and $0\leq k\leq m-2$,}\\
\displaystyle \sum_{p=0}^k [D_p]_{\D}\cdot q^{k(2k+3)-2p(2k-2m+1)}\cdot \binom{m-p-1}{k-p}_{q^4} & \tn{for $n$ odd and $0\leq k \leq m-1$.}
\end{cases}
$$
\end{lemma}

\begin{proof}
Assume that $n$ is even, and write $\sum_{j\geq 0}[H^j_{\overline{O}_k}(S)]_{\D}\cdot q^j=\sum_{p=0}^k [Q_p]_{\D}\cdot P_{m,k,p}(q)$, where $P_{m,k,p}(q)\in \mathbb{Z}[q]$ are described in Lemma \ref{Qdecomp}. It follows that
$$
\sum_{0\leq j\leq d} \left[H^{d-j}_{\overline{O}_k}(S)\right]_{\D}\cdot q^{j}=\sum_{p=0}^k [Q_p]_{\D}\cdot q^d\cdot P_{m,k,p}(q^{-1}).
$$
By the binomial identity (\ref{binomInvert}), it follows that 
$$
q^d\cdot P_{m,k,p}(q^{-1})=q^{d-2(m-k)^2+(m-k)-4(k-p)-4(k-p)(m-k-2)}\cdot \binom{m-p-2}{k-p}_{q^4}=q^{k(2k+3)-4p(k-m+1)}\cdot \binom{m-p-2}{k-p}_{q^4},
$$
completing the proof for $n$ even. The proof for $n$ odd is similar and uses \cite[Main Theorem]{raicu2016local}.
\end{proof}

\section{The structure of local cohomology with support in Pfaffian varieties of even-sized skew-symmetric matrices}\label{Local}
Throughout this section, let $n=2m$ be even, let $W$ be an $n$-dimensional complex vector space, and let $S=\tn{Sym}(\bw^2 W)$. This section is dedicated to proving Theorem \ref{structure}, the fact that the local cohomology modules $H^j_{\overline{O}_k}(S)$ are direct sums of the modules $Q_0,\cdots , Q_k$. 

\subsection{Flag varieties and the relative setting}\label{relative}

Given an integer $0\leq p\leq n$, write $\mathbb{F}([p,n];W)$ for the variety of partial flags
$$
W_{\bullet}:\;\;W=W_n\twoheadrightarrow W_{n-1}\cdots \twoheadrightarrow W_p\twoheadrightarrow 0,
$$
where each $W_q$ is a quotient of $W$ with $\dim W_q=q$. Given $q\in[p,n]$, denote by $\mathcal{Q}_q$ for the tautological rank $q$ quotient sheaf on $\mathbb{F}([p,n];W)$. The fiber of $\mathcal{Q}_q$ over $W_{\bullet}\in \mathbb{F}([p,n];W)$ is $W_q$. There are natural projection maps
$\mathbb{F}([p,n];W) \to \mathbb{F}([p+1,n];W)$ defined by forgetting $W_p$ from the flag $W_{\bullet}$. When $p\leq n-1$, this map identifies $\mathbb{F}([p,n];W)$ with the projective bundle $\mathbb{P}_{\mathbb{F}([p+1,n];W)}(\mathcal{Q}_{p+1})$, which comes with a tautological surjection $\mathcal{Q}_{p+1}\twoheadrightarrow \mathcal{Q}_p$. The kernel of this map is a line bundle, which we denote by $\mathcal{L}_{p+1}$. When $p>0$, the Picard group of $\mathbb{F}([p,n];W)$ is free of rank $(n-p)$, unless $p=0$, in which case the Picard group has rank $n-1$. For $p\geq 1$, we write the following for the line bundle corresponding to $\mu\in \mathbb{Z}^{n-p}$:
$$
\mathcal{L}^{\mu}=\bigotimes_{i=1}^{n-p} \mathcal{L}_{p+i}^{\otimes \mu_i}.
$$
We now introduce relative versions of the invariant ideals $I_{\underline{z}}\subset S$ and the modules $J_{\underline{z},l}$ (see Section \ref{Perlman}). Given an integer $k$ with $0\leq k\leq m$, we write $\mathbb{F}(k)$ for the partial flag variety $\mathbb{F}([2k,n];W)$, and recall the tautological rank $2k$ quotient sheaf $\mathcal{Q}_{2k}$.  We write $Z(k)$ for the geometric vector bundle on $\mathbb{F}(k)$ associated to the locally free sheaf $\bw^2 \mathcal{Q}_{2k}^{\ast}$, with structure map $q:Z(k)\to \mathbb{F}(k)$. Then $q_{\ast}\mathcal{O}_{Z(k)}=\mathcal{S}^k$, where $\mathcal{S}^k$ is the sheaf of algebras
\begin{equation}
\mathcal{S}^{k}=\textnormal{Sym}_{\mathbb{F}(k)}\left(\bw^2\mathcal{Q}_{2k}\right)=\bigoplus_{\underline{x}\in \mathcal{P}(k)} \bS_{\underline{x}^{(2)}}\mathcal{Q}_{2k},
\end{equation}
where the last equality follows from (\ref{cauchy}). For $\underline{x}\in \mathcal{P}(k)$, let $\mathcal{I}_{\underline{x},k}$ denote the ideal in $\mathcal{S}^{k}$ generated by $\bS_{\underline{x}^{(2)}}\mathcal{Q}_{2k}$, and define for $\mathcal{X}\subseteq \mathcal{P}(k)$ the ideal
\begin{equation}
\mathcal{I}_{\mathcal{X},k}=\sum_{\underline{x}\in \mathcal{X}}\mathcal{I}_{\underline{x},k}.
\end{equation}
These ideals satisfy the analogues of (\ref{charI}) and (\ref{contain}). Define the $\mathcal{S}^{k}$-modules
\begin{equation}
\mathcal{J}_{\underline{z},p,k}=\mathcal{I}_{\underline{z},k}/\mathcal{I}_{\mathfrak{succ}(\underline{z},p,k),k}.
\end{equation}
When $k=m$, then $\mathcal{S}^{k}=S=\tn{Sym}(\bw^2 W)$ is the ring of polynomial functions on the space of $n\times n$ skew-symmetric matrices, and $I_{\underline{x}}=\mathcal{I}_{\underline{x},k}$ and $J_{\underline{z},p}=\mathcal{J}_{\underline{z},p,k}$. 

We will denote by $\textnormal{det}^{(k)}$ the line bundle $\textnormal{det}(\mathcal{Q}_{2k})$. Note that $\mathcal{J}_{\underline{z},p,k}\otimes \textnormal{det}^{(k)}=\mathcal{J}_{\underline{z}+(1^k),p,k}$. Let $\lambda\in \mathbb{Z}^k_{\textnormal{dom}}$. If $\lambda=\underline{x}-(d^k)$ for some $d\geq 0$, $\underline{x}\in \mathcal{P}(k)$, write
\begin{equation}\label{Jtwist}
\mathcal{J}_{\lambda,p,k}=\mathcal{J}_{\underline{x},p,k}\otimes (\textnormal{det}^{(k)})^{\otimes(-d)}.
\end{equation}
For $\lambda\in \mathbb{Z}^k_{\textnormal{dom}}$, $0\leq p\leq k$, and $\mu\in \mathbb{Z}^{n-2k}$ define
\begin{equation}
\mathcal{M}^{k}_{\lambda,p,\mu}=\mathcal{J}_{\lambda,p,k}\otimes_{\mathcal{O}_{\mathbb{F}(k)}}\mathcal{L}^{\mu}.
\end{equation}
For $\underline{y}\in \mathcal{P}(n-2k)$, and $d\geq y_1$, define $\underline{x}\in \mathcal{P}(m)$ by $x_1=\cdots =x_k=d$, and $(x_{k+1},\cdots, x_m)^{(2)}=\underline{y}$. For all $\mathbb{S}_{\underline{z}}\mathcal{Q}_{2k}\subseteq \mathcal{J}_{(d^k),k,k}$, the concatenation $(\underline{z},\underline{y})\in \mathbb{Z}^n$ is dominant. Thus, using \cite[Theorem 2.1(b)]{lHorincz2018iterated} and following the proof of \cite[Lemma 3.2]{genericLocal} we obtain the following isomorphism of $S$-modules:
\begin{equation}\label{pushM}
H^{q}\left(\mathbb{F}(k), \mathcal{M}^{k}_{(d^k),k,\underline{y}}\right)=
\begin{cases}
J_{\underline{x},k} & q=0,\\
0 & \textnormal{otherwise.}
\end{cases}
\end{equation}

\subsection{Reduction to a vanishing statement}
In this subsection, let $X=\bw^2W^{\ast}$ be the space of $n\times n$ complex skew-symmetric matrices and identify $S=\tn{Sym}(\bw^2 W)\cong \C[x_{i,j}]_{1\leq i<j\leq n}$. Let $X_1\subset X$ be the basic open affine subset consisting of skew-symmetric matrices with $x_{1,2}\neq 0$, and let $X'$ be the space of $(n-2)\times (n-2)$ skew-symmetric matrices. Write $O'_k$ for the orbit of rank $k$ matrices in $X'$, and write $Q_p'$ for the $\D_{X'}$-modules analogous to (\ref{Q}). Just as in \cite[Section 2.8]{lHorincz2018iterated}, there is a projection $\pi:X_1\to X'$ satisfying $\pi^{-1}(O_k')=O_{k+1}\cap X_1$ for all $k=0,\cdots, m-1$, and 
\begin{equation}\label{restrictQ}
\pi^{\ast}(Q_p')=(Q_{p+1})|_{X_1}=(Q_{p+1})_{x_{1,2}} \tn{ for all $p=0,\cdots, m-1$},
\end{equation}
where the subscript $x_{1,2}$ denote localization at that variable. The projection $\pi$ comes from applying simultaneous row and column operations to eliminate the first two rows and first two columns of an element in $X_1$, and more details may be found in \cite[Lemma 1.2]{jozefiak1979ideals}. Further, writing $S'=\C[x_{i,j}']_{3\leq i<j\leq n}$ for the coordinate ring of $X'$, we have
\begin{equation}\label{restrictH}
\pi^{\ast}(H^j_{\overline{O}_k'}(S'))=H^j_{\overline{O}_{k+1}\cap X_1}(S_{x_{1,2}})=\left(H^j_{\overline{O}_{k+1}}(S)\right)|_{X_1} \tn{ for all $k=0,\cdots, m-1$, and $j\geq 0$}.
\end{equation}
Using the above discussion, we follow the proof of \cite[Proposition 6.8]{lHorincz2018iterated} via induction on $n$ even. The proof there relies on their Lemma's 6.3 - 6.6 and analogues to the equations (\ref{restrictQ}) and (\ref{restrictH}) above. Versions of those lemmas in the skew-symmetric case follow immediately from the fact that $\tn{mod}_{\tn{GL}}(\D)$ is equivalent to the category of $\tn{GL}_{m}(\C)\times \tn{GL}_{m}(\C)$-equivariant coherent $\D$-modules on the space of $m\times m$ generic matrices \cite[Theorem 5.4, Theorem 5.7]{lHorincz2018categories}, and under this equivalence, each $Q_p$ is identified with the module referred to by $Q_p$ in \cite{lHorincz2018iterated}. Let $U=X\setminus \{0\}$ with open immersion $j:U\hookrightarrow X$. As in the end of the proof of \cite[Proposition 6.8]{lHorincz2018iterated}, we have an exact sequence for all $i\geq 0$:
\begin{equation}\label{importantSequence}
0\longrightarrow H^0_{\{0\}}H^i_{\overline{O}_k}(S)\longrightarrow H^i_{\overline{O}_k}(S)\longrightarrow j_{\ast}j^{\ast}H^i_{\overline{O}_k}(S)\longrightarrow H^1_{\{0\}}H^i_{\overline{O}_k}(S)\longrightarrow 0,
\end{equation}
where $j_{\ast}j^{\ast}H^i_{\overline{O}_k}(S)$ is a direct sum of the modules $Q_0,Q_1,\cdots, Q_{m-1}$. Since $H^0_{\{0\}}H^i_{\overline{O}_k}(S)$ is supported on $\{0\}$, it is a direct sum of copies of $E=Q_0$. By \cite[Lemma 6.6]{lHorincz2018iterated}, to prove Theorem \ref{structure}, it suffices to show that $H^1_{\{0\}}H^i_{\overline{O}_k}(S)=0$ for all $i\geq 0$. We verify this vanishing in the next subsection. 

\subsection{Vanishing of local cohomology for the subquotients $J_{\underline{z},p}$}\label{vanish}

Throughout this subsection, $n$ is assumed to be even: $n=2m$, and we continue to write $X=\bw^2 W^{\ast}$, where $W$ is an $n$-dimensional vector space. We will prove the desired vanishing $H^1_{\{0\}}H^i_{\overline{O}_k}(S)=0$ for all $i\geq 0$. In order to do this, we reduce to showing the vanishing $H^1_{\{0\}}(\textnormal{Ext}^j_S(J_{\underline{x},l},S))=0$ for all $0\leq l \leq m$, for all $j\geq 0$, and $\underline{x}\in \mathcal{P}(m)$ with $x_1=\cdots=x_{l+1}$. Using the notation in Section \ref{Perlman}, the defining ideal of $\overline{O}_k$ is $I_{(k+1)\times 1}$, the ideal of $(2k+2)\times (2k+2)$ Pfaffians of the skew-symmetric matrix of indeterminates $(x_{i,j})$. The system of ideals $\{I_{(k+1)\times e}\}_{e\geq 1}$ is cofinal to the collection $\{I_{(k+1)\times 1}^d\}_{d\geq 1}$. By \cite[Remark 7.9]{iyengar2007twenty} we obtain
\begin{equation}\label{limLoc}
H^1_{\{0\}}H^i_{\overline{O}_k}(S)=H^1_{\{0\}}\left(\lim_{e\to \infty}\tn{Ext}^i_S(S/I_{(k+1)\times e},S)\right)=\lim_{e\to \infty}H^1_{\{0\}}(\tn{Ext}^i_S(S/I_{(k+1)\times e},S)).
\end{equation}

\noindent Therefore, by the discussion in Section \ref{Perlman}, to show that $H^1_{\{0\}}H^i_{\overline{O}_k}(S)=0$, it suffices to prove the following:
\begin{theorem}\label{vanish}
Suppose that $0\leq l\leq m$ and that $\underline{x}\in \mathcal{P}(m)$ with $x_1=\cdots=x_{l+1}$. Then for all $j\geq 0$
$$
H^1_{\{0\}}(\textnormal{Ext}^j_S(J_{\underline{x},l},S))=0.
$$
\end{theorem}

\noindent The purpose of this section is to prove this theorem. Using the notation in Section \ref{relative}, let $\mathbb{F}=\mathbb{F}(l)$, $Z=Z(l)$, and $Y=X \times \mathbb{F}$. Consider the following diagram:
$$
\begin{tikzcd}
Z \arrow[hookrightarrow]{r}{s}\arrow[dr, "\pi"] & Y \arrow[d,"p"]\arrow[r, "q"] & \mathbb{F}\\
& X
\end{tikzcd}
$$
where $s$ is the inclusion, $p$ and $q$ are the projections, and $\pi=p\circ s$. For ease of notation, set $d=x_1$. We define $\underline{y}\in \mathbb{N}_{\tn{dom}}^{n-2l}$ via $\underline{y}=(x_{l+1},\cdots, x_m)^{(2)}$ and let $\mathcal{M}=\mathcal{M}^{l}_{(d^l),l,\underline{y}}$. By (\ref{pushM}) and Grothendieck duality \cite[Theorem 11.1]{duality} we obtain:
$$
\textnormal{RHom}_S(J_{\underline{x},l},S)=\textnormal{RHom}_S(\tn{R}\pi_{\ast}\mathcal{M},S)=\tn{R}\pi_{\ast}(\tn{R}\mathscr{H}\E_Z(\mathcal{M},\pi^{!}S))=\tn{R}\pi_{\ast}(\mathcal{M}^{\ast}\otimes_{\mathcal{O}_Z}\pi^{!}S),
$$
where the last isomorphism follows from the fact that $\mathcal{M}$ is locally free. Following the argument in \cite[Section 4.2]{lHorincz2018iterated}, we obtain
$$
\pi^{!}S=s^{\ast}q^{\ast}(\textnormal{det}(\mathcal{Q}_{2l})^{\otimes (4l-2n)}\otimes \mathcal{L}^{(2-2n+4l,4-2n+4l,\cdots , 0)}),
$$
and tensoring with $\mathcal{M}^{\ast}=s^{\ast}q^{\ast}(\textnormal{det}(\mathcal{Q}_{2l})^{\otimes (-d)}\otimes \mathcal{L}^{-\underline{y}})$, we get that $\mathcal{M}^{\ast} \otimes_{\mathcal{O}_Z}\pi^{!}S=\mathcal{M}^{l}_{\lambda,l, \mu}$, where $\lambda=((4l-2n-d)^l)$, and $\mu=(2-2n+4l-y_1,4-2n+4l-y_2,\cdots ,-y_{2m-2l})$. Putting it all together, we obtain
$$
\textnormal{Ext}^j_S(J_{\underline{x},l}, S)=H^j\left(\mathbb{F}, \mathcal{M}_{\lambda,l,\mu}^{l}\right),
$$
where $\lambda_l\leq \mu_1$ since $d\geq y_1$. Just as in \cite{lHorincz2018iterated}, the modules $H^j(\mathbb{F}, \mathcal{M}_{\lambda,l,\mu}^{l})$ have filtrations with successive quotients $J_{\nu,p}$ for various $\nu\in \mathbb{Z}^m_{\tn{dom}}$ and $p\in \mathbb{N}$. We review some notation from (\ref{Bott}) in order to state the following lemma. Let $s\geq 1$ be an integer and let $\rho=(s-1,s-2,\cdots, 1,0)$. Given $\mu\in \mathbb{Z}^s$, write $\tn{sort}(\mu+\rho)$ for the element of $\mathbb{Z}^s$ obtained by putting the entries of $\mu+\rho$ in weakly decreasing order. We write
\begin{equation}\label{bottTwist}
\tilde{\mu}=\tn{sort}(\mu+\rho)-\rho,
\end{equation}
for the dominant integral weight obtained from $\mu$ by applying Bott's Algorithm (see \cite[Remark 4.1.5]{weyman2003cohomology}). Notice that for $\mu=(2-2n+4l-y_1,4-2n+4l-y_2,\cdots ,-y_{2m-2l})$ as above, $\tilde{\mu}$ satifies $\tilde{\mu}_{2i}=\tilde{\mu}_{2i-1}$ for all $i\leq m-l$.

\begin{lemma}\label{badFilt}
Let $0\leq q\leq l$ and $k\geq 0$, suppose that $\lambda\in \mathcal{Z}^l_{\textnormal{dom}}$ with $\lambda_1=\cdots=\lambda_q$. Let $\mu\in \mathbb{Z}^{n-2l}$, and assume that $\tilde{\mu}_{2i}=\tilde{\mu}_{2i-1}$ for all $i\leq m-l$. Suppose that $\lambda_l\leq \mu_j$ for some $j$. The cohomology group $H^k(\mathbb{F}(l),\mathcal{M}^{l}_{\lambda,q,\mu})$ has composition series with composition factors $J_{\nu,p}$, with $\nu\in \mathbb{Z}^m_{\textnormal{dom}}$, $p \leq q$, and $\nu_1=\cdots =\nu_{p+1}$.
\end{lemma}

The proof is similar to that of \cite[Theorem 4.2]{lHorincz2018iterated}. As the numerology is different, we include the proof here.

\begin{proof}
The reduction to the case where $\lambda$ and $\mu$ are dominant is identical to the reduction in the proof of \cite[Theorem 4.2]{lHorincz2018iterated}. One just needs to verify that the hypotheses on $\mu$ hold throughout the reduction. This is fine however, as $\tilde{\tilde{\mu}}=\tilde{\mu}$. Thus, we may assume $\lambda$ and $\mu$ are dominant. Now we proceed by induction on $l$ and $q$. Write $\mu^{(1/2)}=(\mu_1,\mu_3,\cdots, \mu_{n-2l-1})$ for the element of $\mathbb{Z}^{m-l}$ consisting of the entries of $\mu$ with odd indices. As $\mu_{2i}=\mu_{2i-1}$ for all $i\leq m-l$, it follows that $\mu^{(1/2)}$ is obtained from $\mu$ by dividing the column lengths of the corresponding Young diagram by two.

 If $l=q=0$, then $\mathcal{M}^l_{\lambda,q,\mu}=\mathcal{L}^{\mu}$, so that $H^0(\mathbb{F}(l),\mathcal{M}^l_{\lambda,q,\mu})=\bS_{\mu}W=J_{\mu^{(1/2)},0}$, as $\mu$ is dominant. This completes the base case, so assume that $0\leq q<l$, and consider the projecton map $\pi^{(l-1)}:\mathbb{F}(l-1)\to \mathbb{F}(l)$ defined by forgetting $W_{2l-1}$ and $W_{2l-2}$ from the flag $W_{\bullet}\in \mathbb{F}(l-1)$. Define $\lambda^{-}=(\lambda_1,\cdots, \lambda_{l-1})$ and $\mu^{+}=(\lambda_l,\lambda_l, \mu_1,\cdots, \mu_{n-2l})$. By \cite[Theorem 2.1(c)]{lHorincz2018iterated}, we have
$$
\tn{R}^i\pi^{(l-1)}_{\ast}\left(\mathcal{M}^{l-1}_{\lambda^{-},q,\mu^{+}}\right)=
\begin{cases}
\mathcal{M}^l_{\lambda,q,\mu} & \tn{ if $i=0$,}\\
0 & \tn{otherwise.}
\end{cases}
$$
Thus, $H^k(\mathbb{F}(l), \mathcal{M}^l_{\lambda,q,\mu})=H^k(\mathbb{F}(l-1),\mathcal{M}^l_{\lambda^{-},q,\mu^{+}})$, yielding the desired conclusion by induction on $q$, as $\tilde{\mu}^{+}_{2i}=\tilde{\mu}^{+}_{2i-1}$ for $0\leq i\leq m-l+1$. Now assume that $l=q>0$, so that $\lambda=(d^l)$ for some $d\geq 0$. As $\lambda_l\leq \mu_j$ for some $j$, it follows that $d\leq \mu_1$. Just as in the proof of \cite[Theorem 4.2]{lHorincz2018iterated}, $\mathcal{M}=\mathcal{M}^l_{\lambda,l,\mu}$ has a filtration with filtered pieces $\mathcal{M}_i=\mathcal{M}^l_{\lambda+(i^l),l,\mu}$ for $i=0,\cdots, \mu_1-d$. We obtain a filtration
\begin{equation}\label{cohoFilt}
H^k(\mathbb{F}(l),\mathcal{M})\supseteq H^k(\mathbb{F}(l),\mathcal{M}_1)\supseteq\cdots \supseteq H^k(\mathbb{F}(l),\mathcal{M}_{\mu_1-d}). 
\end{equation}
By (\ref{pushM}) we have $H^k(\mathbb{F}(l),\mathcal{M}_{\mu_1-d})=0$ for $k>0$ and $H^0(\mathbb{F}(l),\mathcal{M}_{\mu_1-d})=J_{\nu,l}$, where $\nu_1=\cdots=\nu_{l+1}=\mu_1$ and $\nu_{l+i}=\mu^{(1/2)}_i$ for $i=2,\cdots, m-l$. In this case, $\mathcal{M}_i/\mathcal{M}_{i+1}=\mathcal{M}^l_{\lambda^i,l-1,\mu^i}$, where $\lambda^i=(d+i)^{l-1}$ and $\mu^i=(d+i,d+i,\mu_1,\cdots, \mu_{n-2l})$. Thus, we have $H^k(\mathbb{F}(l),\mathcal{M}_i/\mathcal{M}_{i+1})=H^k(\mathbb{F}(l),\mathcal{M}^l_{\lambda^i,l-1,\mu^i})$, so that the result holds for $\mathcal{M}$ by induction on $q$ and the filtration (\ref{cohoFilt}).
\end{proof}

Therefore, in order to prove Theorem \ref{vanish} (and thus Theorem \ref{structure}), it suffices to prove the following:

\begin{lemma}
If $0\leq l \leq m$ and $\nu\in \mathbb{Z}^m_{\textnormal{dom}}$ with $\nu_1=\cdots =\nu_{l+1}$, then
$$
H_{\{0\}}^1(J_{\nu,l})=0.
$$
\end{lemma}

\begin{proof}
Let $d=\binom{n}{2}$. By (\ref{Jtwist}) we may assume that $\nu\in \mathcal{P}(m)$. By graded local duality \cite[Theorem 14.4.1]{brodman1998local}, we need to show that $\tn{Ext}_S^{d-1}(J_{\nu,l},S)=0$. Suppose for contradiction that this module is nonzero. By \cite[Theorem 3.2]{perlman2017regularity} there is $\underline{t}\in \mathcal{T}_l(\nu)$ and $\lambda\in W(\nu,l,\underline{t})$ with $\binom{2l}{2}+2|\underline{t}|=1$. Conclude that $l=0$ or $l=1$. If $l=0$, then $\underline{t}=(0)$ and $\binom{2l}{2}+2|\underline{t}|=0=1$, a contradiction. Thus, $l=1$ and $\underline{t}=(0)$. However, as $\underline{t}\in \mathcal{T}_l(\nu)$, it has $t_1=l$, a contradiction.
\end{proof}

\section{Lyubeznik numbers for Pfaffian rings of even-sized skew-symmetric matrices}\label{Even}

Let $n=2m$ be even, let $W$ be an $n$-dimensional complex vector space, and let $S=\tn{Sym}(\bw^2 W)$. We set $d=\dim S=\binom{n}{2}$. In this section we compute the local cohomology modules  $H^{\bullet}_{\{0\}}(Q_p)$, yielding the Lyubeznik numbers for Pfaffian rings of even-sized skew-symmetric matrices. We go on to compute $H^{\bullet}_{\{0\}}(D_s)$, which we use in the next section to compute the Lyubeznik numbers in the odd case.

Choose an isomorphism $S\cong \mathbb{C}[x_{i,j}]_{1\leq i<j\leq n}$ and let $\tn{Pf}\in S$ be the $n\times n$ Pfaffian of the skew-symmetric matrix of indeterminants $(x_{i,j})$. Recall from the introduction the $\D$-modules $\langle \tn{Pf}^{-2k}\rangle_{\D}$. Our strategy for computing the Lyubeznik numbers for Pfaffian rings of even-sized skew-symmetric matrices is described as follows:
\begin{enumerate}
\item We realize the modules $\langle \tn{Pf}^{-2k}\rangle_{\D}$ as a direct limit of $S$-modules of the form $I_{\underline{z}}\otimes_S \mathscr{P}_e$, where $\mathscr{P}_e\subseteq \tn{Frac}(S)$ are the fractional ideals defined in Section \ref{Perlman}.

\item We compute the modules $\tn{Ext}^j_S(I_{\underline{z}}\otimes_S \mathscr{P}_e,S)$ for all pairs $(\underline{z},e)$ appearing in step (1). To do this, we use the previous computations of $\tn{Ext}^j_S(S/I_{\underline{x}},S)$ for $\underline{x}\in \mathcal{P}(m)$ described in Section \ref{Perlman}. By careful use of graded local duality \cite[Theorem 14.4.1]{brodman1998local}, we obtain $H^{\bullet}_{\{0\}}(\langle \tn{Pf}^{-2k}\rangle_{\D})$ for all $0\leq k\leq m-1$.

\item By Theorem \ref{structure}, the local cohomology modules $H^j_{\overline{O}_k}(S)$ decompose as a direct sum of modules $Q_0,Q_1,\cdots , Q_k$. Since $Q_p=S_{\tn{Pf}}/\langle \tn{Pf}^{\; 2(p-m+1)}\rangle_{\D}$ for $p=0,\cdots ,m-1$, we use part (2) to compute $H^j_{\{0\}}(Q_p)$ via the long exact sequence of local cohomology. This yields the Lyubeznik numbers by formula (\ref{lyubeznik}).
\end{enumerate}

We begin by recalling graded local duality \cite[Theorem 14.4.1]{brodman1998local} in our setting. We write $(-)^{\vee}$ for the graded Matlis functor (written $^*D$ in \cite{brodman1998local}). With our grading conventions on $E$ (see Section \ref{Dmodule}), there is a graded local duality isomorphism for a finitely-generated graded $S$-module $M$:
\begin{equation}\label{GLLD}
H^j_{\{0\}}(M)\cong \tn{Ext}^{d-j}_S(M,\mathscr{P}_{n-1})^{\vee},
\end{equation}
for all $j\geq 0$, where $\mathscr{P}_{n-1}\subseteq \tn{Frac}(S)$ is the fractional ideal introduced in Section \ref{Perlman}. For example:
$$
H^d_{\{0\}}(S)\cong \tn{Hom}_S(S,\mathscr{P}_{n-1})^{\vee}\cong \mathscr{P}_{n-1}^{\vee}.
$$
By \cite[Exercise 14.4.2(iv)]{brodman1998local}, we have $S^{\vee}\cong E(-d)$, so that $\mathscr{P}_{n-1}^{\vee}\cong E$ (our $E$ is equal to $^*E(d)$ in \cite{brodman1998local}).

Now we work out part (1) of our strategy. Given an integer $e\geq 1$ and $0\leq k\leq m-1$, write $(m-k)\times e$ for the partition $(e^{m-k},0^k)\in \mathcal{P}(m)$. For each $e\geq 1$ and $0\leq k\leq m-1$, we define the module
\begin{equation}\label{M}
\mathscr{N}_{k,e}=I_{(m-k)\times e}\otimes_S \mathscr{P}_{-e-2k},
\end{equation}
which is isomorphic as a representation of $\tn{GL}$ to $I_{(m-k)\times e}\otimes_{\mathbb{C}} \tn{det}(W)^{\otimes (-e-2k)}$. Notice that $\mathscr{N}_{k,e}$ is a submodule of $\mathscr{N}_{k,e+1}$. Indeed:
$$
\mathscr{N}_{k,e}\subset \mathscr{N}_{k,e+1} \iff I_{(m-k)\times e}\otimes_S \mathscr{P}_1\subseteq I_{(m-k)\times (e+1)}\overset{(\ref{twister})}\iff I_{((e+1)^{m-k},1^k)}\subseteq I_{(m-k)\times (e+1)},
$$
and the last containment follows from (\ref{contain}). Similarly, $\mathscr{N}_{k,e}\subseteq \mathscr{N}_{k+1,e}$ for $0\leq k\leq m-2$.
\begin{lemma}\label{limitpfaff}
Let $k$ be an integer with $0\leq k\leq m-1$. Then
\begin{equation}\label{limitExpression}
\langle \textnormal{Pf}^{-2k} \rangle_{\mathcal{D}}=\lim_{\substack{\longrightarrow \\ e}}\;\;(\mathscr{N}_{k,e}).
\end{equation}
Further, the inclusion map $\langle \tn{Pf}^{-2k}\rangle_{\D}\hookrightarrow \langle \tn{Pf}^{-2k-2}\rangle_{\D}$ is direct limit of the inclusions $\mathscr{N}_{k,e}\subseteq \mathscr{N}_{k+1,e}$.
\end{lemma}

\begin{proof}
It is clear that the limit $\lim_{e\to \infty}(\mathscr{N}_{k,e})$ is an $S$-submodule of $S_{\tn{Pf}}$. As the $\tn{GL}$-module $S_{\tn{Pf}}$ has a multiplicity-free decomposition into irreducibles (\ref{charLocal}), it suffices to show that the two sides of (\ref{limitExpression}) have the same $\tn{GL}$ structure. 

$\subseteq :$ By (\ref{charPf}), the $\mathcal{D}$-module $\langle \textnormal{Pf}^{-2k} \rangle_{\mathcal{D}}$ decomposes as a direct sum of $\bS_{\lambda}W^{\ast}$ with $\lambda_{2i-1}=\lambda_{2i}$ for all $1\leq i \leq m$, and $\lambda_{2k+1}\leq 2k$. Equivalently, $\langle \tn{Pf}^{-2k}\rangle_{\D}$ decomposes into irreducibles $\bS_{\mu}W$ with $\mu_{2i-1}=\mu_{2i}$ for all $1\leq i\leq m$, and $\mu_{n-2k}\geq -2k$. By (\ref{twister}) we have 
\begin{equation}\label{newM}
\mathscr{N}_{k,e}\cong \bigoplus_{\substack{\nu\in \mathbb{Z}^m_{\tn{dom}}\\ \nu \geq (-2k^{m-k},(-e-2k)^k)}} \bS_{\nu^{(2)}}W.
\end{equation}
It follows that if $\bS_{\mu}W\subseteq \langle \textnormal{Pf}^{-2k} \rangle_{\mathcal{D}}$, then $\bS_{\mu}W\subseteq \mathscr{N}_{k,e}$ for some $e\gg 0$.

$\supseteq:$ Now suppose that $\bS_{\mu}W\subseteq \mathscr{N}_{k,e}$ for some $e\geq 1$. By (\ref{newM}), we have that $\mu_{2i-1}=\mu_{2i}$ for all $1\leq i\leq m$ and $\mu_{2k}\geq -2k$. Writing $\bS_{\mu}W=\bS_{\lambda}W^{\ast}$ for $\lambda_i=-\mu_{n-i+1}$, we see that $\bS_{\mu}W\subseteq \mathscr{N}_{k,e}$ implies that $\lambda_{2k+1}\leq 2k$ and $\lambda_{2i-1}=\lambda_{2i-1}$ for all $1\leq i\leq m$. By (\ref{charPf}), it follows that $\bS_{\mu}W=\bS_{\lambda}W^{\ast}\subseteq \langle \tn{Pf}^{-2k}\rangle_{\D}$, as required to complete the proof of the first assertion. 

To prove the second assertion, note that the direct limit (\ref{limitExpression}) is simply the union of the $\mathscr{N}_{k,e}$'s. Since $\mathscr{N}_{k,e}\subseteq \mathscr{N}_{k+1,e}$ for all $e\geq 1$, the result follows.
\end{proof}

We now prove two lemmas to tackle step (2) in our strategy. Part (b) of the first lemma will only be used to prove Theorem \ref{evenSimp}, which is a computation at the end of the section that we will use to compute the Lyubeznik numbers for Pfaffian rings of odd-sized skew-symmetric matrices in the next section. Recall that $d=\dim S=\binom{n}{2}$.

\begin{lemma}\label{techLemma}
Let $1\leq k\leq m-1$, and let $1\leq j\leq d$. The following hold:
\begin{enumerate}[label=(\alph*)]
\item The morphisms
\begin{equation}\label{referEq}
\tn{Ext}^j_S(\mathscr{N}_{k,e+1},S)\longrightarrow \tn{Ext}^j_S(\mathscr{N}_{k,e},S),
\end{equation}
induced by the inclusion $\mathscr{N}_{k,e}\subseteq \mathscr{N}_{k,e+1}$ are surjective for all $e\geq 1$.

\item The morphisms
\begin{equation}
\tn{Ext}^j_S(\mathscr{N}_{k,e},S)\longrightarrow \tn{Ext}^j_S(\mathscr{N}_{k-1,e},S),
\end{equation}
induced by the inclusion $\mathscr{N}_{k-1,e}\subseteq \mathscr{N}_{k,e}$ are zero for all $e\geq 1$.
\end{enumerate}
\end{lemma}

\begin{proof}
(a) By (\ref{twister}), we have $\mathscr{N}_{k,e}=I_{\underline{z}}\otimes_S \mathscr{P}_{-e-2k-1}$, where $\underline{z}=((e+1)^{m-k},1^k)$. Thus, \cite[Proposition III.6.7]{hartshorne2013algebraic} implies
$$
\tn{coker}\left(\tn{Ext}^j_S(\mathscr{N}_{k,e+1},S)\to \tn{Ext}^j_S(\mathscr{N}_{k,e},S)\right ) =\mathscr{P}_{e+2k+1}\otimes_S\; \tn{coker}\left(\tn{Ext}^j_S(I_{(m-k)\times (e+1)},S)\to \tn{Ext}^j_S(I_{\underline{z}},S)\right).
$$
Since $j\geq 1$, we have $\tn{Ext}^j_S(I,S)\cong \tn{Ext}_S^{j+1}(S/I,S)$ for all ideals $I\subseteq S$. Therefore (\ref{cokernel}) implies that the cokernel above is equal to the tensor product of $\mathscr{P}_{e+2k+1}$ and an $S$-module that decomposes as a $\mathbb{C}$-vector space as follows:
\begin{equation}\label{Eqtwo}
\bigoplus_{(\underline{x},p)\in \mathcal{Z}(\underline{z})\setminus \mathcal{Z}((m-k)\times (e+1))} \tn{Ext}^{j+1}_S(J_{\underline{x},p},S).
\end{equation}
Using the notation from Section \ref{Perlman}, we have $\mathcal{Z}((m-k)\times(e+1))$ is the set of $(\underline{x},p)$ such that $p=m-k-1$ and $x_1=\cdots=x_{m-k}\leq e$, and $\mathcal{Z}(\underline{z})$ is the set of $(\underline{x},p)$ with $(\underline{x},p)=(\underline{0},m-1)$ or $p=m-k-1$, $x_1=\cdots =x_{m-k}\leq e$, and $(1^m)\leq \underline{x}$. In particular, since $k\geq 1$,
$$
\mathcal{Z}(\underline{z})\setminus \{(\underline{0},m-1)\}\subseteq \mathcal{Z}((m-k)\times (e+1)).
$$
As $J_{\underline{0},m-1}=S/I_{m\times 1}$, it has dimension $d-1$. Thus, $\tn{Ext}^i_S(J_{\underline{0},m-1},S)$ is nonzero if and only if $i=1$, i.e. $\tn{Ext}^{j+1}_S(J_{\underline{0},m-1},S)$ is nonzero if and only if $j=0$. Conclude that (\ref{Eqtwo}) is zero for all $j\geq 1$. Therefore, the maps (\ref{referEq}) are surjective for all $j\geq 1$, as required to complete the proof of (a).

(b) Similar to the previous part, using \cite[Proposition III.6.7]{hartshorne2013algebraic}, it suffices to show that the maps
\begin{equation}
\tn{Ext}^j_S(I_{(m-k)\times e},S)\longrightarrow \tn{Ext}^j_S(I_{(m-k+1)\times e}\otimes_S \mathscr{P}_2,S),
\end{equation}
induced by the inclusion $I_{(m-k+1)\times e}\otimes_S \mathscr{P}_2 \subseteq I_{(m-k)\times e}$ are zero for all $j\geq 1$. Note that by (\ref{twister}) we have
$$
I_{(m-k+1)\times e}\otimes_S \mathscr{P}_2=I_{\underline{y}},
$$
where $\underline{y}=((2+e)^{m-k+1},2^{k-1})$. As $\tn{Ext}^j_S(I,S)\cong \tn{Ext}^{j+1}_S(S/I,S)$ for all $j\geq 1$ and all ideals $I\subset S$, it suffices to show that the maps
\begin{equation}
\tn{Ext}^i_S(S/I_{(m-k)\times e},S)\longrightarrow \tn{Ext}^i_S(S/I_{\underline{y}},S),
\end{equation}
induced by the inclusion $I_{\underline{y}}\subseteq I_{(m-k)\times e}$ are zero for all $i\geq 2$. Using notation from Section \ref{Perlman}, we have $\mathcal{Z}((m-k)\times e)$ is a set of pairs $(\underline{x},p)$ with $p=m-k-1$, and $\mathcal{Z}(\underline{y})$ is a set of pairs $(\underline{x},p)$ with $p=m-1$ or $p=m-k$. As $k\geq 1$, it follows that $\mathcal{Z}((m-k)\times e)$ and $\mathcal{Z}(\underline{y})$ are disjoint. Therefore, by (\ref{kernel}), the maps in question are zero.
\end{proof}

\begin{lemma}\label{injectiveP}
Let $k$ be an integer with $1\leq k\leq m-1$. Then for all $e\geq 1$, and for all $0\leq j< d$, the morphisms
\begin{equation}\label{injMaps}
H^j_{\{0\}}(\mathscr{N}_{k,e})\longrightarrow H^j_{\{0\}}(\mathscr{N}_{k,e+1}),
\end{equation} 
induced by the inclusion $\mathscr{N}_{k,e}\subseteq \mathscr{N}_{k,e+1}$, are injective. In particular for $0\leq j<d$, the multiplicity of $E$ in the $\D$-module $H^j_{\{0\}}(\langle \tn{Pf}^{-2k}\rangle_{\D})$ is equal to the multiplicity of the representation $\tn{det}(W^{\ast})^{\otimes (n-1)}$ in $H^j_{\{0\}}(\mathscr{N}_{k,e})$ for $e\gg 0$.
\end{lemma}

\begin{proof}
Let $\mathcal{K}^j_{k,e}$ be the kernel of the map (\ref{injMaps}) for $0\leq j<d$, and recall the graded Matlis functor $(-)^{\vee}$ \cite[Chapter 13, Chapter 14]{brodman1998local}. As $(-)^{\vee}$ is exact, graded local duality (\ref{GLLD}) implies that 
\begin{align*}
\mathcal{K}^j_{k,e} & =\left( \tn{coker}\left(\tn{Ext}^{d-j}_S(\mathscr{N}_{k,e+1},\mathscr{P}_{n-1})\longrightarrow \tn{Ext}^{d-j}_S(\mathscr{N}_{k,e},\mathscr{P}_{n-1})\right)\right)^{\vee}\\
& = \left( \tn{coker}\left(\tn{Ext}^{d-j}_S(\mathscr{N}_{k,e+1},S)\longrightarrow \tn{Ext}^{d-j}_S(\mathscr{N}_{k,e},S)\right)\otimes_S \mathscr{P}_{n-1}\right)^{\vee}.
\end{align*}
By Lemma \ref{techLemma}(a), conclude that $\mathcal{K}^j_{k,e}=0$ for all $e\geq 1$ and all $0\leq j<d$, as required.

To prove the second assertion, note that $H^j_{\{0\}}(\langle \tn{Pf}^{-2k}\rangle_{D})$ is a $\D$-module supported at the origin, so it must be a direct sum of copies of $E$. As $E$ contains the subrepresentation $\tn{det}(W^{\ast})^{\otimes (n-1)}$ with multiplicity one (since $E$ is the simple $\D$-module $D_0$, this follows from (\ref{characters})), it follows that the multiplicity of $E$ in $H^j_{\{0\}}(\langle \tn{Pf}^{-2k}\rangle_{D})$ is equal to the multiplicity of $\tn{det}(W^{\ast})^{\otimes (n-1)}$ in the representation $H^j_{\{0\}}(\langle \tn{Pf}^{-2k}\rangle_{D})$. Since filtered direct limits commute with $H^j_{\{0\}}(-)$, Lemma \ref{limitpfaff} and the first assertion of this lemma imply that, for $0\leq j<d$, the module $H^j_{\{0\}}(\langle \tn{Pf}^{-2k}\rangle_{D})$ is the union over $e$ of the modules $H^j_{\{0\}}(\mathscr{N}_{k,e})$. Therefore, the multiplicity of $E$ in the $\D$-module $H^j_{\{0\}}(\langle \tn{Pf}^{-2k}\rangle_{\D})$ is equal to the multiplicity of the representation $\tn{det}(W^{\ast})^{\otimes (n-1)}$ in $H^j_{\{0\}}(\mathscr{N}_{k,e})$ for $e\gg 0$.
\end{proof}

Before we complete step (2) in our strategy for computing the Lyubeznik numbers in this case, we store the following lemma. The proof is identical to the proof of \cite[Lemma 6.9]{lHorincz2018iterated}, replacing $\tn{det}$ with $\tn{Pf}$.

\begin{lemma}\label{Pfvanish}
For all $j\geq 0$ and $k<m$ we have $H^j_{\overline{O}_k}(S_{\tn{Pf}})=0$.
\end{lemma}

\begin{lemma}\label{locPfaff}
Let $k$ be an integer with $0\leq k \leq m-1$. Then
$$
\sum_{j\geq 0} \left[H^j_{\{0\}}(\langle \textnormal{Pf}^{-2k} \rangle_{\mathcal{D}})\right]_{\D}\cdot q^j =[E]_{\D}\cdot q^{m(2m-1)-k(2k+3)-4(m-k-1)k}\cdot \binom{m-1}{m-k-1}_{q^4}.
$$
\end{lemma}

\begin{proof}
If $k=0$, then $\langle \tn{Pf}^{-2k}\rangle_{\D}=S$, so that $H^{\bullet}_{\{0\}}(\langle \tn{Pf}^{-2k}\rangle_{\D})=H^d_{\{0\}}(S)=E$. As $\binom{m-1}{m-k-1}_{q^4}=1$ and 
$$
m(2m-1)-k(2k+3)-4(m-k-1)k=m(2m-1)=d,
$$
the result holds for $k=0$.

Now assume that $1\leq k\leq m-1$. We claim that $H^d_{\{0\}}(\langle \tn{Pf}^{-2k}\rangle_{\D})=0$. By (\ref{Q}), there is a short exact sequence
$$
0\longrightarrow \langle \tn{Pf}^{-2k}\rangle_{\D}\longrightarrow S_{\tn{Pf}}\longrightarrow Q_{m-k-1}\longrightarrow 0.
$$
By Lemma \ref{Pfvanish} we have $H^j_{\{0\}}(S_{\tn{Pf}})=0$ for all $j\geq 0$, so the long exact sequence of local cohomology yields $H^0_{\{0\}}(\langle \tn{Pf}^{-2k}\rangle_{\D})=0$ and 
$$
H^j_{\{0\}}(\langle \tn{Pf}^{-2k}\rangle_{\D})\cong H^{j-1}_{\{0\}}(Q_{m-k-1})\;\;\;\tn{for all $j\geq 1$}.
$$
Suppose for contradiction that $H^d_{\{0\}}(\langle \tn{Pf}^{-2k}\rangle_{\D})\neq 0$, so that $H^{d-1}_{\{0\}}(Q_{m-k-1})\neq 0$. By Theorem \ref{structure}, we have $H^c_{\overline{O}_{m-k-1}}(S)=Q_{m-k-1}$, where $c=\tn{codim}\;\overline{O}_{m-k-1}$. Thus, $H^{d-1}_{\{0\}}H^c_{\overline{O}_{m-k-1}}(S)\neq 0$, i.e. the Lyubeznik number $\lambda_{d-1,d-c}(R^{m-k-1})$ is nonzero. Since $k\geq 1$, (\ref{dimension}) implies that $d-1>d-c$, so that by \cite[Properties 3.2(1)]{nunez2016survey}, we obtain a contradiction. We conclude $H^d_{\{0\}}(\langle \tn{Pf}^{-2k}\rangle_{\D})=0$.

We now compute the multiplicity of $E$ in $H^j_{\{0\}}(\langle \tn{Pf}^{-2k}\rangle_{\D})$ for all $0\leq j<d$ and all $1\leq k\leq m-1$. By Lemma \ref{injectiveP} we need to determine the number of copies of $\tn{det}(W^{\ast})^{\otimes (n-1)}$ in $H^j_{\{0\}}(\mathscr{N}_{k,e})$ for all $j\geq 0$ and $e\gg 0$. By (\ref{GLLD}), this is equivalent to determining the number of copies of $\textnormal{det}(W^{\ast})^{\otimes(e+2k)}$ in $\textnormal{Ext}^{d-j}_S(I_{(m-k)\times e},S)$ for large $e$. 

Note that the vanishing locus of $I_{(m-k)\times e}$ is $\overline{O}_{m-k-1}$, a variety of dimension $<d-1$ by (\ref{dimension}). Thus, $\tn{Ext}^1_S(S/I_{(m-k)\times e},S)=0$, so that $\tn{Hom}_S(I_{(m-k)\times e},S)\cong S$. Since $S=\tn{Sym}(\bw^2 W)$, it does not contain the representation $\textnormal{det}(W^{\ast})^{\otimes(e+2k)}$ for all $e\geq 1$. Therefore, since $\textnormal{Ext}^{j}_S(S/I_{(m-k)\times e},S)\cong \textnormal{Ext}^{j-1}_S(I_{(m-k)\times e},S)$ for $j\geq 2$, we obtain:
$$
\sum_{j\geq 0}\left\langle \textnormal{Ext}^j_S(S/I_{(m-k)\times e},S), \tn{det}(W^{\ast})^{\otimes (n+e-2(m-k))}\right\rangle\cdot q^j=\sum_{j\geq 0} \left\langle \textnormal{Ext}^{j-1}_S(I_{(m-k)\times e},S), \tn{det}(W^{\ast})^{\otimes (n+e-2(m-k))}\right\rangle \cdot q^{j}.
$$
By Lemma \ref{multQuot} with $a=m-k$ and $b=e$, we obtain for $e\gg 0$:
\begin{equation}\label{blahp}
\sum_{j\geq 0}\left\langle \textnormal{Ext}^j_S(I_{(m-k)\times e},S), \tn{det}(W^{\ast})^{\otimes (n+e-2(m-k))}\right\rangle\cdot q^j=q^{(2k+1)(k+1)-1}\cdot \binom{m-1}{m-k-1}_{q^4}.
\end{equation}
Write $g_{m,k}(q)\in \mathbb{Z}[q]$ for the polynomial in (\ref{blahp}), so that $\sum_{j \geq 0} [H^j_{\{0\}}(\langle \textnormal{Pf}^{-2k} \rangle_{\mathcal{D}})]_{\D}\cdot q^j=[E]_{\D}\cdot q^d\cdot g_{m,k}(q^{-1})$. Using the binomial identity (\ref{binomInvert}) we get the desired result.
\end{proof}

Finally, we compute the modules $H^j_{\{0\}}(Q_p)$ for all $p=0,\cdots, m-1$, completing step (3) in our strategy to compute the Lyubeznik numbers in this case:

\begin{theorem}\label{locQ}
Let $n=2m$ be even, and let $0\leq p \leq m-1$ be an integer. Then
$$
\sum_{j\geq 0} \left[H^j_{\{0\}}(Q_p)\right]_{\D}\cdot q^j=[E]_{\D}\cdot q^{p(2p+3)}\cdot \binom{m-1}{p}_{q^4}.
$$
\end{theorem}

\begin{proof}
The short exact sequence arising from the inclusion $\langle \tn{Pf}^{\;2(p-m+1)}\rangle_{\D}\subseteq S_{\tn{Pf}}$ yields a long exact sequence of local cohomology $H^{\bullet}_{\{0\}}(-)$. As $Q_p=S_{\tn{Pf}}/\langle \tn{Pf}^{\;2(p-m+1)}\rangle_{\D}$, the result follows from Lemma \ref{Pfvanish} and Lemma \ref{locPfaff}.
\end{proof}

We may now compute the Lyubeznik numbers in the case of even-sized skew-symmetric matrices:

\begin{proof}
[Proof of Theorem \ref{main} for $n=2m$ even] First assume that $k=m-1$, so that $H^1_{\overline{O}_k}(S)=S_{\tn{Pf}}/S=Q_{m-1}$ and $H^j_{\overline{O}_k}(S)=0$ otherwise. By (\ref{lyubeznik}) and Theorem \ref{locQ} we obtain that $L_{m-1}(q,w)=(q\cdot w)^{d-1}$, where $d=\binom{n}{2}$. Alternatively, this follows from \cite[Example 4.2]{nunez2016survey}.

 For $k<m-1$, the computations follow from (\ref{lyubeznik}), Lemma \ref{swapOrder}, and Theorem \ref{locQ}.
\end{proof}

We store the following work for use in the next section:

\begin{lemma}\label{zeroes}
For all $0\leq k\leq m-2$ and $j\geq 0$, the maps
\begin{equation}\label{lemmamorph}
H^j_{\{0\}}(\langle \tn{Pf}^{-2k}\rangle_{\D})\longrightarrow H^j_{\{0\}}(\langle \tn{Pf}^{-2k-2}\rangle_{\D}),
\end{equation}
induced by the inclusion $\langle \tn{Pf}^{-2k}\rangle_{\D}\subseteq \langle \tn{Pf}^{-2k-2}\rangle_{\D}$ are zero.
\end{lemma}

\begin{proof}
We first treat the case $k=0$, so $\langle \tn{Pf}^{-2k}\rangle_{\D}\cong S$. In this case, $H^{\bullet}_{\{0\}}(S)=H^d_{\{0\}}(S)=E$. By Lemma \ref{locPfaff}, we have 
$$
H^d_{\{0\}}(\langle \tn{Pf}^{-2k-2}\rangle_{\D})=H^d_{\{0\}}(\langle \tn{Pf}^{-2}\rangle_{\D})=0.
$$
Therefore, the result holds for $k=0$.

Now assume that $1\leq k \leq m-2$. By Lemma \ref{locPfaff} we have that $H^d_{\{0\}}(\langle \tn{Pf}^{-2k}\rangle_{\D})=0$, so that we only need to consider $j$ with $0\leq j\leq d-1$. For such a $j$, Lemma \ref{injectiveP} implies that the morphisms (\ref{lemmamorph}) are the direct limit over $e\geq 1$ of the maps
\begin{equation}\label{hoope}
H^j_{\{0\}}(\mathscr{N}_{k,e})\longrightarrow H^j_{\{0\}}(\mathscr{N}_{k+1,e}),
\end{equation}
induced by the inclusion $\mathscr{N}_{k,e}\subseteq \mathscr{N}_{k+1,e}$. Write $\mathcal{Y}^j_{k,e}$ for the kernel of the map (\ref{hoope}). Since the graded Matlis dual $(-)^{\vee}$ is exact, it follows from graded local duality (\ref{GLLD}) that
\begin{align*}
\mathcal{Y}^j_{k,e} & = \left(\tn{coker}\left( \tn{Ext}^{d-j}_S(\mathscr{N}_{k+1,e},\mathscr{P}_{n-1})\longrightarrow \tn{Ext}^{d-j}_S(\mathscr{N}_{k,e},\mathscr{P}_{n-1})\right)\right)^{\vee}\\
& = \left(\tn{coker}\left( \tn{Ext}^{d-j}_S(\mathscr{N}_{k+1,e},S)\longrightarrow \tn{Ext}^{d-j}_S(\mathscr{N}_{k,e},S)\right)\otimes_S \mathscr{P}_{n-1}\right)^{\vee}.
\end{align*}
By Lemma \ref{techLemma}(b) and graded local duality (\ref{GLLD}), it follows that $\mathcal{Y}^j_{k,e}=H^j_{\{0\}}(\mathscr{N}_{k,e})$ for all $0\leq j\leq d-1$. Therefore, the maps (\ref{lemmamorph}) are zero for all $j\geq 0$, as required.
\end{proof}

\begin{theorem}\label{evenSimp}
If $s$ is an integer with $0\leq s\leq m$, then 
$$
\sum_{j\geq 0}\left[ H^j_{\{0\}}(D_s)\right]_{\D}\cdot q^j=[E]_{\D}\cdot q^{s(2s-1)}\cdot \binom{m}{s}_{q^4}.
$$
\end{theorem}

\begin{proof}
When $s=0$, we have $D_s=E$. As $H^{\bullet}_{\{0\}}(E)=H^0_{\{0\}}(E)=E$, the result holds in this case. Now assume $s\geq 1$ and consider the short exact sequence arising from the filtration (\ref{filtration}):
$$
0\longrightarrow \langle \tn{Pf}^{-2(m-s-1)}\rangle_{\D}\longrightarrow \langle \tn{Pf}^{-2(m-s)}\rangle_{\D}\longrightarrow D_s\longrightarrow 0.
$$
By Lemma \ref{zeroes}, the long exact sequence of $H^{\bullet}_{\{0\}}(-)$ arising from this short exact sequence splits up into exact sequences:
$$
0\longrightarrow H^j_{\{0\}}(\tn{Pf}^{-2(m-s)}\rangle_{\D})\longrightarrow H^j_{\{0\}}(D_s)\longrightarrow H^{j+1}_{\{0\}}(\tn{Pf}^{-2(m-s-1)}\rangle_{\D})\longrightarrow 0.
$$
for all $j\geq 0$. We conclude that
$$
\sum_{j\geq 0}\left[ H^j_{\{0\}}(D_s)\right]_{\D}\cdot q^j=\sum_{j\geq 0} \left[H^j_{\{0\}}(\langle \tn{Pf}^{-2(m-s)}\rangle_{\D})\right]_{\D}\cdot q^j +\sum_{j\geq 0} \left[H^{j+1}_{\{0\}}(\langle \tn{Pf}^{-2(m-s-1)}\rangle_{\D})\right]_{\D}\cdot q^j.
$$
As $\sum_{j\geq 0} [H^{j}_{\{0\}}(\langle \tn{Pf}^{-2(m-s-1)}\rangle_{\D})]_{\D}\cdot q^j=\sum_{j\geq 0} [H^{j+1}_{\{0\}}(\langle \tn{Pf}^{-2(m-s-1)}\rangle_{\D})]_{\D}\cdot q^{j+1}$, Lemma \ref{locPfaff} implies that
$$
\sum_{j\geq 0}\left[ H^j_{\{0\}}(D_s)\right]_{\D}\cdot q^j = [E]_{\D}\cdot \left( q^{s(2s-1)}\cdot \binom{m-1}{s-1}_{q^4}+q^{s(2s+3)}\cdot \binom{m-1}{s}_{q^4} \right).
$$
We get the desired result after applying the binomial identity (\ref{binomIdent}).
\end{proof}

\section{Lyubeznik numbers for Pfaffian rings of odd-sized skew-symmetric matrices}\label{oddSect}

Throughout this section, let $n=2m+1$ be odd, let $W\cong \mathbb{C}^n$, and let $X=\bw^2 W^{\ast}$ be the space of $n\times n$ skew-symmetric matrices. For ease of notation, we set $V=W^{\ast}$. We continue to write $S=\tn{Sym}(\bw^2W)$ for the ring of polynomial functions on $X$ and $\tn{GL}=\tn{GL}(W)$. Let $\mathbb{G}=\tn{Gr}(2m,V)$ be the Grassmannian of $2m$-dimensional subspaces of $V$, and let $\mathcal{R}$ be the tautological subsheaf of $V\otimes_{\mathbb{C}}\mathcal{O}_{\mathbb{G}}$, a locally free sheaf of rank $2m$. We write $Y$ for the geometric vector bundle associated to the locally free sheaf $\bw^2 \mathcal{R}$, with structure map $q:Y\to \mathbb{G}$. Consider the following diagram:
$$
\begin{tikzcd}
Y \arrow[r, hook, "s"] \arrow[dr, "\pi"] & X \times \mathbb{G} \arrow[d, "p"]\\
& X
\end{tikzcd}
$$
where $s$ is the inclusion, $p$ is the projection, and $\pi=p\circ s$. In this case, $\pi(Y)=X$ and $\pi^{-1}(O_m)\cong O_m$.

Our strategy for computing the local cohomology modules $H^j_{\{0\}}(D_p)$ for $n$ odd is described as follows: Over a basic open affine subset $U\subseteq \mathbb{G}$, the bundle $Y$ trivializes to $\bw^2\mathbb{C}^{2m}\times U$. We will show that there exist $m+1$ simple $\tn{GL}$-equivariant $\D_Y$-modules $D^Y_0,D^Y_1,\cdots, D^Y_m$ such that:
\begin{enumerate}
\item for all $p=0,\cdots, m$, the restrictions to the trivializations $D^Y_p|_{q^{-1}(U)}$ are isomorphic to $D_p^{2m}\otimes_{\mathbb{C}}\mathbb{C}[U]$, where $\mathbb{C}[U]$ is the coordinate ring of $U$ and $D_p^{2m}$ is the simple $\tn{GL}_{2m}(\mathbb{C})$-equivariant $\D$-module on $\bw^2\mathbb{C}^{2m}$ with support $\overline{O}_p\subseteq \bw^2\mathbb{C}^{2m}$,

\item the derived direct image $\tn{R}\pi_{\ast}D^Y_p$ has cohomology in a single degree, isomorphic to $D_p$, where $D_p$ is the simple $\D_X$-module with support $\overline{O}_p\subseteq X$,

\item using fact (1), the calculations of $H^j_{\{0\}}(D_p^{2m})$ in Theorem \ref{evenSimp} may be glued together to obtain $\mathscr{H}^j_{\pi^{-1}(0)}(D^Y_p)$, where $\mathscr{H}^j_Z$ are the derived functors of the functor $\mathscr{H}^0_Z$ of sections with support in a subvariety $Z\subseteq Y$.
\end{enumerate}
As there is an isomorphism of functors $\pi_{\ast}\circ \mathscr{H}^0_{\pi^{-1}(0)}=\mathscr{H}^0_{\{0\}}\circ \pi_{\ast}$, we use facts (1)-(3) to compute $H^j_{\{0\}}(D_p)$ in Theorem \ref{locOdd}, yielding the Lyubeznik numbers for Pfaffian rings of odd-sized skew-symmetric matrices.

\subsection{Equivariant $\D$-modules on $Y$}
We begin by establishing notation for basic open affine subsets of $\mathbb{G}$, following \cite[Section 3.2.2]{eisenbud20163264}. Given a one-dimensional subspace $L$ of $V$, we obtain an open subset of $\mathbb{G}$:
\begin{equation}
U_L=\left\{ H\in \mathbb{G}\mid L\cap H=0\right\}\subseteq \mathbb{G}.
\end{equation}
If we fix $H\in U_L$, then $U_L$ is identified with $\tn{Hom}_{\mathbb{C}}(H,L)$, and under this identification, $H$ is sent to the zero map. For ease of notation throughout, we write $Y_L=q^{-1}(U_L)$ and $O_k^H$ for the orbits of $\bw^2 H$.

The following proposition addresses Step (1) above.

\begin{proposition}\label{DY}
The group $\tn{GL}$ acts on $Y$ with $m+1$ orbits $O^Y_0,\cdots , O_m^Y$. Further, for all $k=0,\cdots ,m$ and all basic open affine subsets $U_L=\tn{Hom}_{\mathbb{C}}(H,L)$, we have 
$$
O^Y_k\cap Y_L=O^H_k\times U_L\subseteq \bw^2 H\times U_L.
$$
As a consequence, there exists $m+1$ simple $\tn{GL}$-equivariant $\mathcal{D}_Y$-modules: $D_0^Y,\cdots , D_m^Y$, each satisfying:
\begin{equation}
q_{\ast}D^Y_s=\bigoplus_{\lambda\in \mathcal{B}(m-s,2m)}\mathbb{S}_{\lambda}\mathcal{R},
\end{equation}
where $\mathcal{B}(m-s,2m)$ is the set of dominant weights defined in Section \ref{Dmodule}.
\end{proposition}

\begin{proof}
For $v\in V$ and $g\in \tn{GL}$, we simply write $gv$ for the dual the action of $g$ on $v$. Note that for all $\varphi\in \tn{Hom}_{\mathbb{C}}(H,L)$, the fibers of $\mathcal{R}$ are given by
$$
\mathcal{R}_{\varphi}=\tn{Image}\left(H \xlongrightarrow{[\tn{id}\;\varphi]} V\right).
$$
Therefore, under the identification $U_{L}=\tn{Hom}_{\mathbb{C}}(H, L)$ we have 
\begin{equation}\label{trivialize}
Y_L=\left\{\left(\sum (v_i\wedge w_i+\varphi(v_i)\wedge w_i + v_i\wedge \varphi(w_i)),\varphi\right)\in X\times U_L\mid v_i,w_i\in H\right\}.
\end{equation}
Using this description, there is an isomorphism $Y_L\xrightarrow{\sim} \bw^2 H\times U_L$ via the map
\begin{equation}\label{localTriv}
(v\wedge w+\varphi(v)\wedge w + v\wedge \varphi(w),\varphi)\mapsto (v\wedge w, \varphi).
\end{equation}
The action of $g\in \tn{GL}$ on $Y$ sends $(v\wedge w+\varphi(v)\wedge w + v\wedge \varphi(w),\varphi)\in Y_L$ to $(gv\wedge gw+g\varphi(v)\wedge gw + gv\wedge g\varphi(w),g\varphi g^{-1})$ in $Y_{gL}$. Thus, via the identification (\ref{localTriv}), $g$ sends $(v\wedge w, \varphi)\in \bw^2H \times U_{L}$ to $(gv\wedge gw, g\varphi g^{-1})\in \bw^2 (gH)\times U_{gL}$. Since $\tn{GL}$ acts transitively on $\mathbb{G}$, the first two assertions of the proposition follow from the fact that the orbits of $\bw^2 H$ under the action of $\tn{GL}(H)\subseteq \tn{GL}$ are $O_0^H,\cdots , O_m^H$.

For $s=0,\cdots, m$, we write $D_s^Y$ for the simple $\tn{GL}$-equivariant $\mathcal{D}_Y$-module corresponding to the trivial local system on the orbit $O^Y_s$ via the Riemann-Hilbert correspondence (see \cite[Theorem 11.6.1]{hotta2007d}). By the second assertion of the proposition, we know how each $D^Y_s$ restricts to a trivialization $Y_L$:
\begin{equation}
D^Y_s|_{Y_L}=D^H_s\otimes_{\mathbb{C}} \mathbb{C}[U_{L}],\;\;\;\;\;\;\; s=0,\cdots ,m,
\end{equation}
where $D^H_s$ is the simple $\D_{\bw^2H}$-module corresponding to the trivial local system on $O^H_s$, and $\mathbb{C}[U_{L}]$ is the coordinate ring of the affine space $U_L$. As $\pi(Y)=X$ and $\pi^{-1}(O_{m})\cong O_{m}$, there is an open immersion $j:O_{m}\hookrightarrow Y$. In fact, $j(O_{m})=O^Y_{m}$, so that $j(O_{m})$ is locally defined by the nonvanishing of the $2m\times 2m$ Pfaffian on the open affine sets $\bw^2H\times U_{L}$. We write $S^H$ for the coordinate ring of $\bw^2H$ and let $\tn{Pf}\in S^H$ denote the $2m\times 2m$ Pfaffian. By the above discussion, it follows that we have the following description of $j_{\ast}(\mathcal{O}_{O_m})$ restricted to $Y_L$:
\begin{equation}
j_{\ast}(\mathcal{O}_{O_{m}})|_{Y_L}=S^H_{\tn{Pf}}\otimes_{\mathbb{C}}\mathbb{C}[U_{L}].
\end{equation}
Thus, $j_{\ast}(\mathcal{O}_{O_{m}})|_{Y_L}$ has filtration as in (\ref{filtration}), so by \cite[Corollary 1.4.17(ii)]{hotta2007d}, the $\mathcal{D}_Y$-module $j_{\ast}\mathcal{O}_{O_{m}}$ has filtration with composition factors $D_0^Y,\cdots ,D_m^Y$, each with multiplicity one.
Since $\tn{Sym}_{\mathbb{G}}(\bw^2 \mathcal{R}^{\ast})$ is the direct image via $q$ of $\mathcal{O}_Y$, we conclude from (\ref{charLocal}) that 
\begin{equation}\label{relativeDecomp}
q_{\ast}\left( j_{\ast}\mathcal{O}_{O_{m}}\right)=\bigoplus_{\lambda\in \mathbb{Z}^m_{\tn{dom}}} \bS_{\lambda^{(2)}} \mathcal{R},\;\;\;\;\tn{and}\;\;\;\;q_{\ast}D^Y_s=\bigoplus_{\lambda\in \mathcal{B}(m-s,2m)} \bS_{\lambda} \mathcal{R},
\end{equation}
where the set $\mathcal{B}(m-s,2m)$ was introduced in Section \ref{Dmodule}.
\end{proof}

We are now ready to prove the following, resolving step (2) in our strategy to compute the local cohomology modules $H^j_{\{0\}}(D_p)$ for $n$ odd.

\begin{lemma}\label{pushEven}
For all $0\leq p\leq m$ we have 
$$
\tn{R}^{2m-2p}\pi_{\ast} D^Y_p\cong D_p,
$$
and $\tn{R}^i\pi_{\ast} D^Y_p=0$ otherwise. In particular, we have that $\tn{R}^{2m-2p}\pi_{\ast}(j_{\ast}\mathcal{O}_{O_m})\cong D_p$ and $\tn{R}^i\pi_{\ast}(j_{\ast}\mathcal{O}_{O_m})=0$ otherwise.
\end{lemma}

\begin{proof}
We begin by showing that $\tn{R}^{2m-2p}\pi_{\ast} D^Y_p$ has the same $\tn{GL}$-structure as $D_p$, and that $\tn{R}^j\pi_{\ast} D^Y_p=0$ for $j\neq 2m-2p$. Since $X$ is affine, we have $\tn{R}^i\pi_{\ast}D^Y_p\cong H^i\left(Y,D^Y_p\right)$, and since $q:Y\to \mathbb{G}$ is an affine morphism, we have $H^i(\mathbb{G},q_{\ast}D^Y_p)\cong H^i\left(Y,D^Y_p\right)$. Therefore, using Proposition \ref{DY} we obtain for all $i\geq 0$:
\begin{equation}\label{bigIso}
\tn{R}^i\pi_{\ast}D^Y_p\cong\bigoplus_{\lambda\in \mathcal{B}(m-p,2m)}H^i(\mathbb{G},\bS_{\lambda}\mathcal{R}).
\end{equation}

Given $\lambda\in \mathcal{B}(m-p,2m)$, write 
\begin{equation}
\lambda^{\ast}=(-\lambda_{2m},-\lambda_{2m-1},\cdots, -\lambda_2,-\lambda_1),
\end{equation}
so that $\bS_{\lambda}\mathcal{R}=\bS_{\lambda^{\ast}}\mathcal{R}^{\ast}$.
 let $\gamma=\gamma_{\lambda}=(\lambda^{\ast},0)\in \mathbb{Z}^{2m+1}$ and let $\rho=(2m,2m-1,\cdots, 0)\in \mathbb{Z}^{2m+1}$. Using the notation from (\ref{Bott}), we have
\begin{equation}\label{Bott2}
H^j(\mathbb{G}, \bS_{\lambda}\mathcal{R})=
\begin{cases}
\bS_{\tilde{\gamma}}W & \tn{ if $\gamma+\rho$ has distinct entries and $j=\sigma$,}\\

0 & \tn{ otherwise.}
\end{cases}
\end{equation}
As $\lambda\in \mathcal{B}(m-p,2m)$, we know that $\lambda_{2m-2p}\geq 2m-2p-1$, $\lambda_{2m-2p+1}\leq 2m-2p$, and $\lambda_{2i}=\lambda_{2i-1}$ for all $i$. Thus, $\lambda^{\ast}_{2p+1}\leq -2m+2p+1$ and $\lambda_{2p}^{\ast}\geq -2m+2p$, so that $(\gamma+\rho)_{2p}\geq 1$ and $(\gamma+\rho)_{2p+1}\leq 1$. As $\lambda_{2i}=\lambda_{2i-1}$ for all $i$, conclude that $(\gamma+\rho)_{2p+2}=(\gamma+\rho)_{2p+1}-1\leq 0$. If $(\gamma+\rho)_{2p+1}=1$, then $(\gamma+\rho)_{2p+2}=0$. Since $(\gamma+\rho)_{2m+1}=0$, it follows that $\gamma+\rho$ has repeated entries. Therefore, $H^j(\mathbb{G},\bS_{\lambda}\mathcal{R})=0$ in this case. Similarly, when $(\gamma+\rho)_{2p+1}=0$, we have $H^j(\mathbb{G},\bS_{\lambda}\mathcal{R})=0$.

Now assume that $(\gamma+\rho)_{2p+1}\leq -1$. Since $(\gamma+\rho)_{2p}\geq 1$, it follows that sorting $\gamma+\rho$ requires $2m-2p$ transpositions and
$$
\tilde{\gamma}_{\lambda}=(\lambda_1^{\ast},\cdots, \lambda_{2p}^{\ast},2p-2m,\lambda_{2p+1}^{\ast}+1,\cdots, \lambda_{2m}^{\ast}+1).
$$
Notice that if we reverse the order of $\tilde{\gamma}$ and multiply by $-1$, we get a unique element of $\mathcal{B}(m-p,2m+1)$. In other words $(\tilde{\gamma})^{\ast}\in \mathcal{B}(m-p,2m+1)$. Conclude that there is a bijection:
\begin{equation}
\left\{ \lambda\in \mathcal{B}(m-p,2m)\mid \lambda^{\ast}_{2p+1}\geq 2p-2m-1\right\}\;\; \leftrightarrow \;\;\mathcal{B}(m-p,2m+1),
\end{equation}
defined by sending $\lambda\in \mathcal{B}(m-p,2m)$ to $(\tilde{\gamma}_{\lambda})^{\ast}$. Since $\bS_{\lambda}W=\bS_{\lambda^{\ast}}W^{\ast}$, it follows from (\ref{bigIso}) and (\ref{Bott2}) that $\tn{R}^j\pi_{\ast}D^Y_p=0$ for $j\neq 2m-2p$ and
\begin{equation}
\tn{R}^{2m-2p}\pi_{\ast}D^Y_p=\bigoplus_{\lambda\in \mathcal{B}(m-p,2m+1)}\bS_{\lambda}W^{\ast}.
\end{equation}
By (\ref{characters}), it follows that $\tn{R}^{2m-2p}\pi_{\ast}D^Y_p$ has the same $\tn{GL}$-structure as the simple $\D_X$-module $D_p$.

Recall that $j_{\ast}\mathcal{O}_{O_m}$ has composition factors $D_0^Y,\cdots, D_m^Y$, each with multiplicity one. As $2m-2p$ is even and $2m-2p\neq 2m-2s$ for all $p\neq s$, we conclude that $\tn{R}^{2m-2p}\pi_{\ast}(j_{\ast} \mathcal{O}_{O_m})\cong \tn{R}^{2m-2p}\pi_{\ast}D^Y_p$ for all $0\leq p\leq m$, and $\tn{R}^i\pi_{\ast}(j_{\ast} \mathcal{O}_{O_m})=0$ otherwise. As $\pi$ is a $\tn{GL}$-equivariant morphism, to complete the proof of both assertions we only need to show that the higher direct images $\tn{R}^{2m-2p}\pi_{\ast}D^Y_p$ are actually $\D_X$-modules. In fact, we will show that $\tn{R}\pi_{\ast}(j_{\ast}\mathcal{O}_{O_m})$ is equal to $\int_{\iota}\mathcal{O}_{O_m}$, where $\iota:O_m\hookrightarrow X$ is the inclusion and $\int_{\iota}$ is the $\mathcal{D}$-module pushforward.

Since $\iota$ is an open immersion, \cite[Example 1.5.22]{hotta2007d} implies that $\int_{\iota}=\tn{R}{\iota}_{\ast}$. As $j$ is affine and $\pi\circ j=\iota$, we obtain $\tn{R}\pi_{\ast}\circ j_{\ast}=\tn{R}\iota_{\ast}$. Therefore, there is an isomorphism 
\begin{equation}\label{conclusion}
\int_{\iota}\mathcal{O}_{O_{m}}\cong \tn{R}\pi_{\ast}(j_{\ast}\mathcal{O}_{O_{m}}),
\end{equation}
which shows that each $\tn{R}^i\pi_{\ast}(j_{\ast}\mathcal{O}_{O_{m}})$ is a $\mathcal{D}_X$-module, completing the proof.
\end{proof}

\begin{remark}
Lemma \ref{pushEven} actually recovers the local cohomology of $S$ with support in the submaximal Pfaffian variety $\overline{O}_{m-1}$. Indeed, writing $\iota:O_m\hookrightarrow X$ for the open immersion of matrices of maximal rank, there is an exact sequence
$$
0\longrightarrow H^0_{\overline{O}_{m-1}}(S)\longrightarrow S\longrightarrow \iota_{\ast} \mathcal{O}_{O_m}\longrightarrow H^1_{\overline{O}_{m-1}}(S)\longrightarrow 0,
$$
and $H^i_{\overline{O}_{m-1}}(S)\cong \tn{R}^{i-1}\iota_{\ast}\mathcal{O}_{O_m}$ for all $i\geq 2$. As $\tn{R}\iota_{\ast}=\tn{R}\pi_{\ast}\circ j_{\ast}$, Lemma \ref{pushEven} yields the $\mathcal{D}$-module structure of the desired local cohomology modules. The $\tn{GL}$-structure of these modules is originally due to Raicu-Weyman-Witt \cite[Theorem 5.5]{raicu2014submax}, and the $\mathcal{D}$-module structure first appeared in \cite[Main Theorem]{raicu2016local}.
\end{remark}

\subsection{Local cohomology computations and Lyubeznik numbers in the odd-sized case}

In this subsection, we compute the modules $H^j_{\{0\}}(D_p)$ for all $j\geq 0$ and all $0\leq p\leq m$, obtaining the Lyubeznik numbers in the case of odd-sized skew-symmetric matrices. Given a closed subvariety $Z$ of $Y$ or $X$, write $\mathscr{H}^0_Z$ for the functor of sections with support in $Z$, and write $\Gamma_{\mathcal{D}_Y}$ for the Grothendieck group of the category of coherent $\mathcal{D}_Y$-modules.  As we will see in the following lemma, the local cohomology modules $\mathscr{H}^j_{\pi^{-1}(0)}(D^Y_p)$ are determined by Theorem \ref{evenSimp}:

\begin{lemma}\label{relLocal}
We have the following in $\Gamma_{\mathcal{D}_Y}$:
$$
\sum_{j\geq 0} \left[\mathscr{H}^j_{\pi^{-1}(0)}(D^Y_p)\right]_{\D_Y}\cdot q^j=[D^Y_0]_{\D_Y}\cdot q^{p(2p-1)}\cdot \binom{m}{p}_{q^4}.
$$
\end{lemma}

\begin{proof}
Using notation as above, fix $L \in G(1,V)$ and $H\in \mathbb{G}$ such that $L\cap H=0$. Lemma \ref{DY} implies that $D^Y_p|_{Y_L}\cong D^H_p\otimes_{\mathbb{C}}\mathbb{C}[U_L]$, and $\mathscr{H}^j_{\pi^{-1}(0)}(D^Y_p)|_{Y_{L}}\cong \mathscr{H}^j_{\pi^{-1}(0)\cap Y_L}(D^Y_p|_{Y_{L}})$, so we may study the problem locally. Recall that $S^H$ is the coordinate ring of $\bw^2 H$. Since $\pi^{-1}(0)\cap Y_{L}=\{0\}\times U_{L}$, \cite[Proposition 7.15(3)]{iyengar2007twenty} yields
\begin{equation}\label{pil}
\Gamma\left(Y_L,\mathscr{H}^j_{\pi^{-1}(0)\cap Y_{L}}\left(D^Y_p|_{Y_{L}}\right)\right)\cong H^j_{\{0\}\times U_{L}}\left(D^H_p\otimes_{\mathbb{C}}\mathbb{C}[U_{L}]\right)\cong H^j_{\{0\}}(D_p^H)\otimes_{\mathbb{C}} \mathbb{C}[U_{L}].
\end{equation}
By (\ref{pil}) and Theorem \ref{evenSimp} the lemma follows.
\end{proof}

\begin{theorem}\label{locOdd}
If $n=2m+1$ is odd and $0\leq p\leq m$, then
$$
\sum_{j\geq 0}\left[H^j_{\{0\}}(D_p)\right]_{\D}\cdot q^j= [E]_{\D} \cdot q^{p(2p+1)}\cdot \binom{m}{p}_{q^4}.
$$
\end{theorem}

\begin{proof}
There is an isomorphism of functors $\pi_{\ast}\circ \mathscr{H}^0_{\pi^{-1}(0)}= \mathscr{H}^0_{\{0\}}\circ\pi_{\ast}$, yielding the isomorphism in the derived category $\tn{R}\pi_{\ast}\circ \tn{R}\mathscr{H}^0_{\pi^{-1}(0)}=\tn{R}\mathscr{H}^0_{\{0\}}\circ\tn{R}\pi_{\ast}$. Thus, there is a spectral sequence for all $0\leq p\leq m$:
\begin{equation}\label{spect}
\tn{R}^i\pi_{\ast}\left(\mathscr{H}^j_{\pi^{-1}(0)}(D^Y_p) \right)\Rightarrow \tn{H}^{i+j}\left(\tn{R}\mathscr{H}^0_{\{0\}}(\tn{R}\pi_{\ast}D^Y_p)\right).
\end{equation}
By Lemma \ref{pushEven}, $\tn{R}^j\pi_{\ast}D^Y_0$ is zero unless $j=2m$. Thus (\ref{spect}) is degenerate, so Lemma \ref{relLocal} implies that
\begin{equation}\label{firstSpec}
\sum_{j\geq 0}\left[\tn{H}^j\left(\tn{R}\mathscr{H}^0_{\{0\}}(\tn{R}\pi_{\ast}\mathscr{D}^Y_p)\right)\right]_{\D}\cdot q^j=[E]_{\D}\cdot q^{p(2p-1)+2m}\cdot \binom{m}{p}_{q^4}.
\end{equation}
Next, there is a spectral sequence 
\begin{equation}\label{now}
H^i_{\{0\}}\left(\tn{R}^j\pi_{\ast}D^Y_p\right)\Rightarrow \tn{H}^{i+j}\left(\tn{R}\mathscr{H}^0_{\{0\}}(\tn{R}\pi_{\ast}D^Y_p)\right).
\end{equation}
By Lemma \ref{pushEven}, we have that $\tn{R}^j\pi_{\ast}D^Y_p$ is zero unless $j=2m-2p$. Thus, the spectral sequence (\ref{now}) is degenerate, so by Lemma \ref{pushEven} and (\ref{firstSpec}) we obtain
$$
\sum_{j\geq 0} \left[H^j_{\{0\}}(D_p)\right]_{\D}\cdot q^j=[E]_{\D}\cdot q^{p(2p-1)+2m-(2m-2p)}\cdot \binom{m}{p}_{q^4}=[E]_{\D}\cdot q^{p(2p+1)}\cdot \binom{m}{p}_{q^4},
$$
as required.
\end{proof}

We end by explaining how to obtain the Lyubeznik numbers in the case of odd-sized skew-symmetric matrices. As the category of $\tn{GL}(W)$-equivariant coherent $\D$-modules on $X$ is semi-simple, the formula in Lemma \ref{swapOrder} completely describes the indecomposable summands of the local cohomology modules $H^j_{\overline{O}_k}(S)$ for all $k=0,\cdots,m-1$. We conclude that the odd case in Theorem \ref{main} holds using (\ref{lyubeznik}), Lemma \ref{swapOrder}, and Theorem \ref{locOdd}.

\section*{Acknowledgments}
The author is very grateful to Claudiu Raicu for his guidance while this work was done. We thank the support of the National Science Foundation Graduate Research Fellowship under Grant No. DGE-1313583.

\bibliographystyle{alpha}
\bibliography{mybib}

\newcommand{\etalchar}[1]{$^{#1}$}
\begin{thebibliography}{NBWZ16}

\bibitem[BS98]{brodman1998local}
MP~Brodmann and RY~Sharp.
\newblock Local cohomology.
\newblock {\em Cambridge Studies in Advanced Mathematics}, 60, 1998.

\bibitem[EH16]{eisenbud20163264}
David Eisenbud and Joe Harris.
\newblock {\em 3264 and all that: A second course in algebraic geometry}.
\newblock Cambridge University Press, 2016.

\bibitem[Har66]{duality}
Robin Hartshorne.
\newblock {\em Residues and duality}.
\newblock Lecture notes of a seminar on the work of A. Grothendieck, given at
  Harvard 1963/64. With an appendix by P. Deligne. Lecture Notes in
  Mathematics, No. 20. Springer-Verlag, Berlin-New York, 1966.

\bibitem[Har77]{hartshorne2013algebraic}
Robin Hartshorne.
\newblock {\em Algebraic Geometry (Graduate Texts in Math. Vol. 52)}.
\newblock Springer, 1977.

\bibitem[HTT07]{hotta2007d}
Ryoshi Hotta, Kiyoshi Takeuchi, and Toshiyuki Tanisaki.
\newblock {\em D-modules, perverse sheaves, and representation theory}, volume
  236.
\newblock Springer Science \& Business Media, 2007.

\bibitem[ILL{\etalchar{+}}07]{iyengar2007twenty}
Srikanth Iyengar, Graham~J Leuschke, Anton Leykin, Ezra Miller, and Claudia
  Miller.
\newblock {\em Twenty-four hours of local cohomology}, volume~87.
\newblock American Mathematical Soc., 2007.

\bibitem[JP79]{jozefiak1979ideals}
Tadeusz J{\'o}zefiak and Piotr Pragacz.
\newblock Ideals generated by {P}faffians.
\newblock {\em Journal of Algebra}, 61(1):189--198, 1979.

\bibitem[LR18]{lHorincz2018iterated}
Andr{\'a}s~C L{\H{o}}rincz and Claudiu Raicu.
\newblock Iterated local cohomology groups and {L}yubeznik numbers for
  determinantal rings.
\newblock {\em arXiv preprint arXiv:1805.08895}, 2018.

\bibitem[LW19]{lHorincz2018categories}
Andr{\'a}s~C L{\H{o}}rincz and Uli Walther.
\newblock On categories of equivariant {D}-modules.
\newblock {\em Advances in Mathematics}, 351:429--478, 2019.

\bibitem[Lyu93]{lyubeznik1993finiteness}
Gennady Lyubeznik.
\newblock Finiteness properties of local cohomology modules (an application of
  {D}-modules to commutative algebra).
\newblock {\em Inventiones mathematicae}, 113(1):41--55, 1993.

\bibitem[NBWZ16]{nunez2016survey}
Luis N{\'u}{\~n}ez-Betancourt, Emily~E Witt, and Wenliang Zhang.
\newblock A survey on the {L}yubeznik numbers.
\newblock {\em Mexican Mathematicians Abroad}, 657:137, 2016.

\bibitem[Per17]{perlman2017regularity}
Michael Perlman.
\newblock Regularity and cohomology of {P}faffian thickenings.
\newblock {\em arXiv preprint arXiv:1711.02777}, 2017.

\bibitem[Rai16]{raicu2016characters}
Claudiu Raicu.
\newblock Characters of equivariant $\mathcal{D}$-modules on spaces of
  matrices.
\newblock {\em Compositio Mathematica}, 152(9):1935--1965, 2016.

\bibitem[RW14]{genericLocal}
Claudiu Raicu and Jerzy Weyman.
\newblock Local cohomology with support in generic determinantal ideals.
\newblock {\em Algebra Number Theory}, 8(5):1231--1257, 2014.

\bibitem[RW16]{raicu2016local}
Claudiu Raicu and Jerzy Weyman.
\newblock Local cohomology with support in ideals of symmetric minors and
  {P}faffians.
\newblock {\em J. Lond. Math. Soc. (2)}, 94(3):709--725, 2016.

\bibitem[RWW14]{raicu2014submax}
Claudiu Raicu, Jerzy Weyman, and Emily~E. Witt.
\newblock Local cohomology with support in ideals of maximal minors and
  sub-maximal {P}faffians.
\newblock {\em Adv. Math.}, 250:596--610, 2014.

\bibitem[Swi15]{switala2015lyubeznik}
Nicholas Switala.
\newblock Lyubeznik numbers for nonsingular projective varieties.
\newblock {\em Bulletin of the London Mathematical Society}, 47(1):1--6, 2015.

\bibitem[Wey03]{weyman2003cohomology}
Jerzy Weyman.
\newblock {\em Cohomology of vector bundles and syzygies}, volume 149.
\newblock Cambridge University Press, 2003.

\end{thebibliography}

\Addresses

\end{document}